\documentclass[12pt]{amsart}
\usepackage[english]{babel}
\usepackage[utf8]{inputenc}
\usepackage{lipsum}
\usepackage{mathrsfs}
\usepackage{stix}
\usepackage{fullpage}
\usepackage{mathtools}
\usepackage{amsmath}
\usepackage{tikz-cd}
\usepackage[colorlinks]{hyperref}
\usepackage[T1]{fontenc}
\usepackage{fullpage}
\usepackage{mathtools}
\usepackage{amsmath}
\usepackage{tikz-cd}
\numberwithin{equation}{section}
\theoremstyle{plain}
\newtheorem{theorem}{Theorem}[section]

\newtheorem{cor}[theorem]{Corollary}

\newtheorem{lemma}[theorem]{Lemma}
\newtheorem{rem}[theorem]{Remark}

\def\X{\mathbb{Y}}
\allowdisplaybreaks
\begin{document}
\title[Functional inequalities on symmetric spaces of noncompact type]{Functional inequalities on symmetric spaces of noncompact type and applications}

\author[A. Kassymov]{Aidyn Kassymov}
\address{
  Aidyn Kassymov:
  \endgraf
  Department of Mathematics: Analysis, Logic and Discrete Mathematics
  \endgraf
  Ghent University, Belgium
  \endgraf
  and  
  \endgraf
  Institute of Mathematics and Mathematical Modeling
  \endgraf
  Almaty, Kazakhstan
  \endgraf
  and
  \endgraf
  Al-Farabi Kazakh National University
  \endgraf
  Almaty, Kazakhstan
  \endgraf
  {\it E-mail address} {\rm kassymov@math.kz} and {\rm aidyn.kassymov@ugent.be}}

\author[V. Kumar]{Vishvesh Kumar}
\address{
  Vishvesh Kumar:
 \endgraf
  Department of Mathematics: Analysis, Logic and Discrete Mathematics
  \endgraf
  Ghent University, Belgium
  \endgraf
  {\it E-mail address} {\rm vishvesh.kumar@ugent.be} and {\rm vishveshmishra@gmail.com}
  }

\author[M. Ruzhansky]{Michael Ruzhansky}
\address{Michael Ruzhansky: \endgraf Department of Mathematics: Analysis, Logic and Discrete Mathematics \endgraf Ghent University, Ghent, Belgium \endgraf and \endgraf 
School of Mathematical Sciences\endgraf Queen Mary University of London, United Kingdom \endgraf
{\it E-mail address} {\rm michael.ruzhansky@ugent.be}}

\thanks{The authors wish to thank the anonymous referee for his/her helpful comments and suggestions that helped to improve the quality of the paper. The authors are grateful to Jean-Philippe Anker for inspiring this work by his minicourse on symmetric spaces at Ghent University, and by further stimulating discussions. 
 The authors would also like to thank Hong-Wei Zhang for several fruitful discussions.  AK, VK and MR are supported  by the FWO Odysseus 1 grant G.0H94.18N: Analysis and Partial Differential Equations and by the Methusalem programme of the Ghent University Special Research Fund (BOF) (Grant number 01M01021). AK is supported by Science Committee of Ministry of Science and Higher Education of the Republic of Kazakhstan (Grant No. BR20281002).  VK and MR are supported by FWO Senior Research Grant G011522N.  MR is also supported  by EPSRC grant EP/V005529/1. }


\begin{abstract}   The aim of this paper is to begin a systematic study of functional inequalities on symmetric spaces of noncompact type of higher rank. Our first main goal of this study is to establish the Stein-Weiss inequality, also known as a weighted Hardy-Littlewood-Sobolev inequality, for the Riesz potential on symmetric spaces of noncompact type. This is achieved by performing delicate estimates of ground spherical function with the use of polyhedral  distance on symmetric spaces and by combining the integral Hardy inequality developed by Ruzhansky and Verma with the sharp Bessel-Green-Riesz kernel estimates on symmetric spaces of noncompact type obtained by Anker and Ji. As a consequence of the Stein-Weiss inequality, we  deduce Hardy-Sobolev, Hardy-Littlewood-Sobolev, Gagliardo-Nirenberg and Caffarelli-Kohn-Nirenberg inequalities on symmetric spaces of noncompact type. The second main purpose of this paper is to show the applications of aforementioned inequalities for studying nonlinear PDEs on symmetric spaces. Specifically, we show that the Gagliardo-Nirenberg inequality can be used  to   establish small data global existence results for the semilinear wave equations with damping and mass terms for the  Laplace-Beltrami operator on symmetric spaces.
\end{abstract}
\maketitle
%\tableofcontents
\allowdisplaybreaks
\section{Introduction}
The study of functional inequalities and weighted functional inequalities plays a significant role in the investigation of problems in differential geometry, harmonic analysis, partial differential equations and in several other areas of mathematics (see \cite{He99,GLS97,Be08}). In particular, these inequalities have been  utilised extensively to study global wellposedness results related to several important non-linear partial differential equations. In this paper, we establish several important functional inequalities including Stein-Weiss inequality, Hardy-Sobolev, Gagliardo-Nirenberg, and Caffarelli-Kohn-Nirenberg inequalities on the higher rank Riemannian symmetric spaces of noncompact type and present some applications to study global existence of wave equations with dumping and mass terms associated to the Laplace-Beltrami operator on Riemannian symmetric spaces. Riemannian symmetric spaces represent a significant category of Riemannian manifolds that are non-positively curved and encompass hyperbolic spaces. An intriguing characteristic of Riemannian symmetric spaces is that each of them can be expressed as $G/K$ for some noncompact, connected, semisimple Lie group $G$ with a finite center, and its maximal subgroup $K$ (see \cite{Hel94}). This characteristic enables the utilisation of representation theory and consequently, Fourier analysis on semisimple Lie groups in the study of analysis of symmetric spaces (\cite{GV88}).

As in the Euclidean space setting, establishing certain functional inequalities on Riemannian manifolds is interesting in itself and is tightly useful in the analysis of nonlinear partial differential equations (PDEs) \cite{He99}. The study of the best constants of functional inequalities on manifolds has led to the conclusion of many geometrical and topological properties of underlying manifolds. We refer to several interesting papers \cite{A76,DoX04,T76,Lie83,Carr,DD,Kri18,KI09} and references therein for more details. In this work, our focus is on establishing certain interesting functional inequalities on Riemannian symmetric spaces of noncompact type, which are useful for studying nonlinear PDEs on such spaces. Recently, many researchers have contributed to the development of certain important functional inequalities on non-positively curved manifolds using different suitable methods from Fourier analysis and geometric analysis. However, most of them were confined to rank one symmetric spaces, such as real or complex hyperbolic spaces and their different models (see \cite{Be15} and references therein). In this case, one uses Helgason-Fourier analysis on hyperbolic spaces \cite{Hel94,GV88}.

One fundamental functional inequality in the Euclidean harmonic analysis is the classical Stein-Weiss inequality established by Stein and Weiss \cite{StWe58}. It states that:
\begin{theorem}\label{Classiacal_Stein-Weiss_inequality}
Let $0<\lambda<N$, $1<p<\infty$, $\alpha<\frac{N(p-1)}{p}$, $\beta<\frac{N}{q}$, $\alpha+\beta\geq0$ and $\frac{\lambda-\alpha -\beta}{N}=\frac{1}{p}-\frac{1}{q}$. For $1<p\leq q<\infty$, we have 
\begin{equation} \label{swin}
\||x|^{-\beta}I_{\lambda}u\|_{L^{q}(\mathbb{R}^{N})}\leq C \||x|^{\alpha}u\|_{L^{p}(\mathbb{R}^{N})},
\end{equation}
where $C$ is a positive constant independent of $u$. Here the Riesz potential $I_\lambda$ on $\mathbb{R}^N$ is defined 
as
\begin{equation}\label{NDoper}
I_{\lambda}u(x)=\int_{\mathbb{R}^{N}}\frac{u(y)}{|x-y|^{N-\lambda}}dy,\,\,\,\,0<\lambda<N.
\end{equation}
\end{theorem}

The unweighted version of inequality \eqref{swin} was proved by Sobolev \cite{Sob38} by extending a multi-dimensional version of the Hardy-Littlewood inequality \cite{HL28}. For this reason the unweighted case of \eqref{swin} is called  the Hardy-Littlewood-Sobolev inequality. 
In other words, the Stein-Weiss inequality is a  radially weighted version of the Hardy-Littlewood-Sobolev inequality.

In the Euclidean space setting, many researchers have studied generalisations of the Hardy-Littlewood-Sobolev and the Stein-Weiss inequalities. There are several works devoted to the analysis of the best constants and extremisers of the aforementioned inequalities. The literature of this area is so vast that it is practically impossible to fully review it, but we cite \cite{CF74, FM, JD, Lie83, MW, Per, DD, Be95, Be08,Be75} here just to mention a very few of them. In the noncommutative setting (e.g. Heisenberg groups, homogeneous groups, general Lie groups), the Hardy-Littlewood-Sobolev and the Stein-Weiss inequalities  are also well-developed, see \cite{FS74, FL12, HLZ, KRS, RY, RY1, RSbook, CLT19, BPTV, Be15, Be21}. We refer to \cite{HZ16} for a version of the Hardy-Littlewood-Sobolev inequality on compact Riemannian manifolds. 

One of the main objectives of this paper is to prove  the Stein-Weiss inequality on symmetric spaces of noncompact type. The strategy of this paper follows that of \cite{RY}, where the case of graded groups was considered, however the argument on symmetric spaces is more geometrically involved. Now, we state this result as follows. 

\begin{theorem}\label{thmSteinWeissintro}
Let $X$ be a symmetric space  of noncompact type of dimension $n\geq 3$.  Let $0<\sigma< n$, $1<p<\infty$,  $\alpha<\frac{n}{p'}$, $\beta<\frac{n}{q},$ $\alpha +\beta\geq0$ and $\frac{\sigma-\alpha -\beta}{n}=\frac{1}{p}-\frac{1}{q}$.  Then, for $1<p\leq q<\infty$ and for sufficiently large $\xi>0,$ there exists a positive constant $C$ independent of $u$, such that
\begin{equation} \label{SWineqintro}
    \left(\int_{X}\left|\int_{X}G_{\xi, \sigma}(y^{-1}x)u(y)dy\right|^{q}\,|x|^{-\beta q}\,dx\right)^{\frac{1}{q}}\leq C\||x|^{\alpha}u\|_{L^{p}(X)}.
\end{equation}

\end{theorem}
Here $G_{\xi, \sigma}$ denotes the Bessel-Green-Riesz Kernel on the Riemannian symmetric space $X$ of any rank  (see Section \ref{preli} for the definition), and $|x|$ denotes the distance between $0:=eK \in X:=G/K$ and the point $x:=gK \in X$. Here $G$ is a noncompact, connected, semisimple Lie group with a finite center and $K$ is a maximal compact subgroup of $G.$  We also clarify here that the $y^{-1}x$ is not defined on $X$ but on the corresponding group $G$ via the identification $X:=G/K.$

As in the Euclidean setting, the Stein-Weiss inequality \eqref{SWineqintro} implies the Hardy-Littlewood-Sobolev inequality on Riemannian symmetric spaces of noncompact type  by choosing $\alpha=0$ and $\beta=0.$ It states that:
\begin{theorem}\label{HLSsymintro}
Let $X$ be a symmetric space of noncompact type of dimension $n\geq3$. Suppose that $0<\sigma<n$, $1<p<\infty$ and $\frac{\sigma}{n}=\frac{1}{p}-\frac{1}{q}$. Then for $1<p<q<\infty$ and $u\in L^{p}(X)$, we have
\begin{equation}
    \left(\int_{X}\left|\int_{X}G_{\xi, \sigma}(y^{-1}x)u(y)dy\right|^{q}dx\right)^{\frac{1}{q}}\leq C\left(\int_{X}|u(x)|^{p}dx\right)^{\frac{1}{p}},
\end{equation}
where $C$ is a positive constant independent of $u$.
\end{theorem}

We remark here that the Hardy-Littlewood-Sobolev inequality on noncompact symmetric spaces of noncompact type was also obtained by Anker \cite{Anker92} (see also \cite[Section 4]{Stri83} and references therein) by different methods.

The Hardy inequality is one of the well-known inequalities of G. H. Hardy which basically says that: 
\begin{equation}\label{dirhar}
\left\|\frac{f}{|x|}\right\|_{L^{p}(\mathbb{R}^{N})}\leq \frac{p}{N-p}\|\nabla f\|_{L^{p}(\mathbb{R}^{N})},\,\,\,\,1<p<N,
\end{equation}
where $f\in C^{\infty}_{0}(\mathbb{R}^{N})$ and  $\nabla$ is the Euclidean gradient. It is known that the constant $\frac{p}{N-p}$ is sharp. 

There is a vast literature available on the Hardy inequalities on Euclidean spaces, Lie groups and on manifolds. We refer to \cite{Carr, DD, He99, KI09, RSbook, KI13, Kri18, KI13, Yaf99, CPR15, RV} and references therein for recent developments. Recently, 
in \cite{BP}, the authors  proved a fractional analogue of the Hardy inequality on the symmetric spaces of noncompact type using the fractional Poisson kernel on symmetric spaces with the help of the extension problem for the fractional Laplace-Beltrami operator on symmetric spaces \cite{BGS}.

One of our main aims is to show  an analogue of the classical multi-dimensional Hardy inequality on symmetric spaces of noncompact type. As a consequence, we obtain the uncertainly principle in our setting. In fact, we first establish the following inequality, an analogue of the Hardy-Sobolev inequality on symmetric spaces of noncompact type.
\begin{theorem}\label{thmHardySobolevintro}
Let $X$ be a symmetric space of noncompact type of dimension $n\geq3$. Suppose that $0<\sigma<n$, $1<p\leq q<\infty$ and $0\leq \beta< \frac{n}{q}$ such that $\frac{\sigma -\beta}{n}=\frac{1}{p}-\frac{1}{q}$.  Then for  $u\in H^{\sigma, p}(X)$, we have
\begin{equation}\label{HardySobolevintro}
    \left\|\frac{u}{|x|^{\beta}}\right\|_{L^{q}(X)}\leq C\|u\|_{H^{\sigma, p}(X)},
\end{equation}
where  $C$ is a positive constant  independent of $u$.
\end{theorem}
Here $H^{\sigma, p}(X)$ denotes the Sobolev space on the symmetric space $X$ defined as in  \eqref{SOS} in the next section.

If we take $q=p$  and $0< \sigma <\frac{n}{p} $   in Theorem \ref{thmHardySobolevintro}, then inequality \eqref{HardySobolevintro} gives the the following Hardy inequality on symmetric spaces
    \begin{equation}\label{Hardyinintro}
    \left\|\frac{u}{|x|^{\sigma}}\right\|_{L^{p}(X)}\leq C\|u\|_{H^{\sigma, p}(X)}.
\end{equation}

The uncertainly principle on symmetric spaces can be derived 
 from \eqref{Hardyinintro}. Indeed, we will show that
    \begin{equation}
    \left\|u\right\|_{L^{2}(X)}\leq C\|u\|_{H^{\sigma, p}(X)}\||x|^\sigma u\|_{L^{\frac{p}{p-1}}(X)}.
\end{equation}

The classical Sobolev inequality (or a continuous Sobolev embedding) is one of prominent functional inequalities widely used to study partial differential equations. Let $\Omega\subset\mathbb{R}^{N}$ be a measurable set and let $1<p<N$. Then the (classical) Sobolev inequality is formulated as
\begin{equation}
\|u \|_{L^{p^*}(\Omega)}\leq C\|\nabla u \|_{L^{p}(\Omega)},\,\,\,u\in C^{\infty}_{0}(\Omega),
\end{equation}
where $C=C(N,p)>0$ is a positive constant, $p^{*}=\frac{Np}{N-p}$ and $ \nabla$ is a standard gradient on $\mathbb{R}^{N}$. The best constant of this inequality was obtained by Talenti in \cite{T76} and Aubin in \cite{A76}. 
   
   E. Gagliardo \cite{Gag} and  L. Nirenberg \cite{Nir} independently, obtained an (interpolation) inequality, widely known as the  Gagliardo-Nirenberg inequality, which says that 
 \begin{equation}
 \|u\|^{p}_{L^{p*}(\mathbb{R}^{N})}\leq C \|\nabla u\|^{N(p^{*}-2)/2}_{L^{2}(\mathbb{R}^{N})}\|u\|^{(2p^{*}-N(p^{*}-2))/2}_{L^{2}(\mathbb{R}^{N})},\,\,u\in H^{1}(\mathbb{R}^{N}),
 \end{equation}
 where \begin{equation*}
 \begin{cases}
   2\leq p^{*}\leq\infty\,\,\text{for}\,\, N = 2,\\
   2\leq p^{*} \leq \frac{2N}{N-2}\,\,\text{for} \,\,N > 2.
 \end{cases}
\end{equation*}
Sobolev and Gagliardo-Nirenberg inequalities  have several applications in PDEs and variational principles. In this paper, we establish the Sobolev inequality and the Gagliardo-Nirenberg inequality on symmetric spaces of noncompact type which we will state as follows.

\begin{theorem}\label{thmGagliardoNirenbergintro}
Let $X$ be a symmetric space of noncompact type of dimension $n\geq3$. Suppose  $0<\sigma<n$,  $\tau>0$, $p>1$, $\sigma p<n$, $\mu\geq1$, $a\in(0,1]$ and 
\begin{equation}
    \frac{1}{\tau}=a\left(\frac{1}{p}-\frac{\sigma}{n}\right)+\frac{(1-a)}{\mu}.
\end{equation}
Then there exists a positive constant $C$ such that 
\begin{equation}
        \|u\|_{L^{\tau}(X)}\leq C \|u\|^{a}_{H^{\sigma,p}}\|u\|^{1-a}_{L^{\mu}(X)}.
\end{equation}
\end{theorem}

By choosing $a=1$ in the above theorem we obtain the Sobolev inequality on symmetric spaces of noncompact type. Indeed,  for $0<\sigma<n$ and $1<p\leq q<\infty$ such that $\frac{\sigma}{n}=\frac{1}{p}-\frac{1}{q}$ we obtain
    \begin{equation}
    \left\|u\right\|_{L^{q}(X)}\leq C\|u\|_{H^{\sigma, p}(X)}.
\end{equation}

 We note that the Sobolev inequality on symmetric spaces and on the general Lie groups was previously obtained by many authors; see \cite{Var89, Stri83,BPTV, RY1}.

We will show applications of the Gagliardo-Nirenberg inequality to prove the small data global existence of the solution for the semilinear PDEs on symmetric spaces of noncompact type. For this, the following particular version of the Gagliardo-Nirenberg inequality will be useful. 
\begin{theorem}
Let $X$ be a symmetric space of noncompact type of dimension $n\geq3$. Let $\tau\in\left[2,\frac{2n}{n-2}\right]$ and $a=\frac{n(\tau-2)}{\tau}.$
Then we have 
\begin{equation}\label{GNPDEin}
    \|u\|_{L^{\tau}(X)}\leq C\|u\|^{a}_{H^{1,2}(X)}\|u\|^{1-a}_{L^{\tau}(X)}.
\end{equation}
\end{theorem}

In one of their pioneering works, L. Caffarelli, R. Kohn and L. Nirenberg in \cite{CKN} established the following inequality:
\begin{theorem}
Let $N\geq1$, and let $l_1$, $l_2$, $l_3$, $a, \, b, \, d,\, \delta \in \mathbb{R}$ be such that $l_1, l_2 \geq 1$,
$l_3 > 0, \,\,0 \leq \delta \leq 1,$ and
\begin{equation}
\frac{1}{l_1}+\frac{a}{N},\,\,\,\frac{1}{l_2}+\frac{b}{N},\,\,\,\frac{1}{l_3}+\frac{\delta d+(1-\delta) b}{N}>0.
\end{equation}
Then,
\begin{equation}
\||x|^{\delta d+(1-\delta) b}u\|_{L^{l_{3}}(\mathbb{R}^{N})}\leq C\||x|^{a}\nabla u\|^{\delta}_{L^{l_{1}}(\mathbb{R}^{N})}\||x|^{b} u\|^{1-\delta}_{L^{l_{2}}(\mathbb{R}^{N})},\,\,\,u\in C^{\infty}_{c}(\mathbb{R}^{N}),
\end{equation}
if and only if
\begin{multline}
\frac{1}{l_3}+\frac{\delta d+(1-\delta) b}{N}=\delta\left(\frac{1}{l_{1}}+\frac{a-1}{N}\right)+
(1-\delta)\left(\frac{1}{l_2}+\frac{b}{N}\right),\\
a-d\geq0,\,\,\,\,\text{if}\,\,\,\delta>0,\\
a-d\leq1,\,\,\,\,\text{if}\,\,\,\delta>0\,\,\,\text{and}\,\,\,\frac{1}{l_3}+\frac{\delta d+(1-\delta) b}{N}=\frac{1}{l_1}+\frac{a-1}{N},
\end{multline}
where $C$ is a positive constant independent of $u$.
\end{theorem}

It is worth noting that the Caffarelli-Kohn-Nirenberg inequality includes aforementioned well-known inequalities, such
as the Gagliardo–Nirenberg inequality, Hardy–Sobolev inequality and Sobolev inequality. We refer to \cite{Ngu22, Kri18} and references therein  for the investigation regarding the effects of the curvature of the Riemannian manifolds for the validity of the Hardy and Caffarelli-Kohn-Nirenberg inequalities on these manifolds and their best constants. 
In this paper, we obtain the following Caffarelli-Kohn-Nirenberg inequality on symmetric spaces of noncompact type.
\begin{theorem}
Let $X$ be a symmetric space of noncompact type  of dimension $n\geq3$ and $\sigma$ be such that $0<\sigma<n$. Suppose    $p>1$, $0<q<\tau<\infty$ such that $a\in\left(\frac{\tau-q}{\tau},1\right]$ and $p\leq \frac{a\tau q}{q-(1-a)\tau}$. Let $b,c$ be real numbers such that $0\leq(c(1-a)-b)\leq \frac{n(q-(1-a)\tau)}{q\tau}$ and $\frac{\sigma-n}{n}-\frac{(c(1-a)-b)q}{an}+\frac{q-(1-a)\tau}{a\tau q}-\frac{1}{p}+1=0$. Then there exists a positive constant independent of $u$ such that
\begin{equation}
    \||x|^{b}u\|_{L^{\tau}(X)}\leq C\|u\|^{a}_{H^{\sigma, p}(X)}\||x|^{c}u\|^{1-a}_{L^{q}(X)}.
\end{equation}
\end{theorem}

As an application of the established Gagliardo-Nirenberg inequality, we show the small data global existence of the solution for the following nonlinear Cauchy problem involving the shifted Laplace-Beltrami operator $\Delta_x:=\Delta+|\rho|^2$ on symmetric space of noncompact type $X,$ where $\rho$ is the half sum of multiplicity of positive roots:
\begin{equation} \label{problem1intro}
\begin{split}
\begin{cases}
 u_{tt}(x,t)-\Delta_{x}u(x,t)+bu_{t}(x,t)+mu(x,t)=f(u(x,t)),\,\,\,\,(x,t)\in X\times (0,T),\,\,T>0,\\
 u(x,0)=u_{0}(x),\,\,\,\,x\in X,\\
 u_{t}(x,0)=u_{1}(x),\,\,\,\,x\in X,
 \end{cases}
 \end{split}
\end{equation}
where $b,m>0$ and $f:\mathbb{R}\rightarrow \mathbb{R}$ satisfies the following conditions:
\begin{equation}\label{f1in}
    f(0)=0
\end{equation}
 and
\begin{equation}\label{f2in}
    |f(u)-f(v)|\leq C(|u|^{p-1}+|v|^{p-1})|u-v|.
\end{equation}

The existence and non-existence results for the semilinear wave equation \eqref{problem1intro} with or without dumping and mass term  have been studied by many prominent researchers for Euclidean spaces  and on certain Lie groups by employing different methods. On the Euclidean space $\mathbb{R}^N,$ the small data global existence of semilinear wave equation associated with Laplacian $\Delta_{{\mathbb{R}^N}}$ on  $\mathbb{R}^n$ (with $b=0$ and $m=0$) is closely related with the {\it Strauss conjecture}. For more details on the Strauss conjecture on the Euclidean space, we cite \cite{Str81, Tat01, Joh79, Kat80, GLS97}.   On the Riemannian manifolds of negative curvature, the small data global existence  for the semilinear wave equation (with $b=0$) \eqref{problem1intro} has been investigated in details during the recent years \cite{MT11, MT12, AP14, APV15, SSW19,SSWZ20, Zhang21, Zhang20, AZ2020}. It is well-known that, in contrast with the Euclidean space,   the Strauss conjecture type phenomena does not occur in the case of negatively curved Riemannian manifolds. Particularly, this was observed in the setting of hyperbolic spaces (also on Damek-Ricci spaces) in \cite{MT11, MT12, AP14, APV15}. Later, these results were extended to non-trapping asymptotically hyperbolic manifolds \cite{SSWZ20} and to the Riemannian manifolds with strictly negatively sectional curvature \cite{SSW19}  using different method. Very recently, Anker and Zhang \cite{AZ2020} investigated  semilinear wave equations on general symmetric spaces of noncompact type  and proved that a similar phenomena holds for general symmetric spaces; see also \cite{Zhang21, Zhang20}. The Strichartz inequality played a very important role for investigating aforementioned results on the Riemannian symmetric spaces  \cite{Zhang20,Zhang21, AZ2020}. 

In this paper, we consider the semilinear dumped wave equation with a mass term  \eqref{problem1intro} on  symmetric spaces of noncompact type and
prove the following result concerning the global existence and uniqueness of problem \eqref{problem1intro}. We use the Gagliardo-Nirenberg inequality \eqref{GNPDEin}  to establish the proof of the next theorem instead of the Strichartz inequality on symmetric spaces.

\begin{theorem}
Let $X$ be a symmetric space of noncompact type  of dimension $n\geq3$ and let  $1\leq p\leq \frac{n}{n-2}$. Suppose that $f$ satisfies the conditions \eqref{f1in}-\eqref{f2in}. Assume that  $u_{0}\in H^{1,2}(X)$ and $u_{1}\in L^{2}(X)$ are such that $\|u_{0}\|_{H^{1,2}(X)}+\|u_{1}\|_{L^{2}(X)}<\varepsilon$. Then there exists $\varepsilon_{0}>0$ such that for all $0<\varepsilon\leq\varepsilon_{0}$ the Cauchy problem \eqref{problem1intro} has a unique global solution  $u\in C(\mathbb{R}_{+}, H^{1,2}(X))\cap C^{1}(\mathbb{R}_{+},L^{2}(X))$.
\end{theorem}

We will organise this manuscript as follows: In the next section, we will provide a brief overview of the analysis on the Riemannian symmetric spaces and the Helgason-Fourier transform. In addition to this,  we will also recall some useful tools such as the Sobolev spaces on symmetric spaces and the integral Hardy inequalities on metric measure spaces. Section \ref{sec3} will be mainly occupied by the proof of the Stein-Weiss inequality on symmetric spaces of any rank.  In Section \ref{sec4}, we will derive several important functional (Hardy, Hardy-Sobolev, Gagliardo-Nirenberg, and Caffarelli-Kohn-Nirenberg) inequalities on symmetric spaces of noncompact type. 
In Section \ref{sec5}, we present an application of the Gagliardo-Nirenberg inequality obtained in Section \ref{sec4} to the global existence results of wave equations with dumping and mass terms associated with the Laplace-Beltrami operator on symmetric spaces on noncompact type.

\section{Riemannian symmetric spaces and Helgason-Fourier transform} \label{preli}

In this section, we recall some basic definitions,  notation and nomenclature related with the higher rank Riemannian symmetric spaces of noncompact type. We  also present definitions and fundamental properties of the Helgason-Fourier transform, Sobolev spaces and some kernel estimates on symmetric symmetric on noncompact type. The material presented in this section can be found in the excellent books and research papers \cite{Hel94, Hel78, GV88, Anker92, AJ99, Stri83}. Finally, we discuss the weighted integral Hardy inequalities on metric measure spaces with general weights \cite{RV} which will be helpful to establish our results in the subsequent sections. 
\subsection{Riemannian symmetric spaces}
Let $G$ be a noncompact, connected, semisimple Lie group with finite center and let $K$ be a maximal compact subgroup. The homogeneous space $X:=G/K$ is a Riemannian symmetric space of noncompact type. Let us assume that $\theta$ is a fixed  Cartan involution on the Lie algbra  $\mathfrak{g}$ of $G$ associated with the Cartan decomposition $\mathfrak{g}=\mathfrak{t} \oplus \mathfrak{p}$ at the Lie algebra level, where $\mathfrak{t}$ and $\mathfrak{p}$ are $+1$ and $-1$ eigenspaces of $\theta$ respectively. It is known that if $\mathfrak{B}$ is the Cartan killing form of $\mathfrak{g}$ then $\mathfrak{B}$ induces the $G$-invariant metric $d$ on $X$ by identifying the tangent space at origin $eK$ of $X$ with $\mathfrak{p}$ and by restricting $\mathfrak{B}$ to $\mathfrak{p}.$ The distance between two points $x_1=g_1K$ and $x_2=g_2K$ of $X$ will be denoted by $d(x_1, x_2).$  We will also use the notation $|x|$ to denote $d(0, x),$ the distance between $0=eK \in X$ and the point $x \in X.$   Let $\mathfrak{a}$ be a maximal abelian subalgebra of  $\mathfrak{p}$ and  let $\mathfrak{a}^*$ be its dual space. The dimension of $\mathfrak{a}$ is called the rank of $X.$ We denote $\dim \mathfrak{a}=l.$  For $\alpha \in \mathfrak{a}^*,$ we define 
$$\mathfrak{g}_\alpha:=\{Y \in \mathfrak{a}: [H, Y]=\alpha(H)Y,\,\,\text{for all}\,\, H \in \mathfrak{a}\}.$$ Then the set of restricted root of $\mathfrak{g}$ with respect to $\mathfrak{a}$ is denoted by $\Sigma$ and defined as $$\Sigma:=\{\alpha \in \mathfrak{a}^* \backslash \{0\}:\mathfrak{g}_\alpha \neq 0 \}.$$ We denote $m_\alpha=\dim (\mathfrak{g}_\alpha)$ for $\alpha \in \mathfrak{a}^*.$ Let us choose a connected component in $\mathfrak{a}$ in a manner that $\alpha \neq 0$ for all $\alpha \in \Sigma.$ 
Denote by $\mathfrak{a}^+$ the connected component, called a positive Weyl chamber. 
Now, with respect to $\mathfrak{a}^+,$ we define positive roots  and positive indivisible roots by $\Sigma^+=\{\alpha \in \Sigma: \alpha>0\,\, \text{on}\,\, \mathfrak{a}^+ \}$ and  $\Sigma^+_0=\{\alpha \in \Sigma: \frac{\alpha}{2} \notin \Sigma^+ \}$ respectively. 
We set $\mathfrak{n}= \oplus_{\alpha \in \Sigma^+} \mathfrak{g}_\alpha.$ Then $\mathfrak{n}$ is a nilpotent subalgebra of $\mathfrak{g}.$
We denote the half sum  of positive roots counted with multiplicities $m_{\alpha}$ by $ \rho := \frac{1}{2} \sum_{\alpha \in \Sigma^+} m_\alpha \alpha. $ The dimension and the pseudo-dimension of $X$ will be denoted by $n$ and $\nu$ respectively, that is, $ n=l+\sum_{\alpha \in \Sigma^+} m_\alpha $ and $\nu= l+2|\Sigma^+_0|.$ The Iwasawa decomposition of $\mathfrak{g}$ is given by $\mathfrak{g}= \mathfrak{t} \oplus \mathfrak{a} \oplus \mathfrak{n}$ on the Lie algebra level. On the Lie group level, if we write $N =\exp \mathfrak{n}$ and $A= \exp \mathfrak{a},$  then we get the Iwasawa decomposition of $G=KAN.$ This means every $g \in G$ can be uniquely written as $g= k(g)\, \exp (H(g)) \,n(g),$ where $k(g) \in K, H(g) \in \mathfrak{a}$ and $n(g) \in N.$ The maps $(k, a, n) \mapsto kan$ is a global diffeomorphism of $K \times A \times N$ onto $G.$ 

Let $\Delta$ be the Laplace-Beltrami operator on $X$ with respect to the $G$-invariant Riemannian metric and $dx$ be the corresponding measure. It is known that the $L^2$-spectrum of $\Delta$ is $(-\infty, -|\rho|^2).$ Let $M$ be the centralizer of $A$ in $K$ and $M'$ be the normalizer of $A$ in $K.$ Then $M$ is the normal subgroup of $M'$ and normalize $N.$ The factor group $W=M'/M$ is a finite group of order $|W|$, called the Weyl group of $X.$ The action of the Weyl group $W$ on $\mathfrak{a}$ is given by an adjoint action. It acts as a group of orthogonal transformations (preserving the Cartan-Killing form) on $\mathfrak{a}^*$ by $(s\lambda)(H)=\lambda(s^{-1}\cdot H)$ for $H \in \mathfrak{a},$ $\lambda \in \mathfrak{a}^*$ and $s \in W,$ where $g.Y=\text{Ad}(g)(Y)$ for $g\in G,\, Y \in \mathfrak{g}.$ We fix a normalized Haar measure on $dk$ on the compact group $K$ and the Haar measure $dn$ on $N.$  We have the decompositions 
\begin{align*}
    \begin{cases}
        \,G\,=\,N\,(\exp\mathfrak{a})\,K 
        \quad&\textnormal{(Iwasawa)}, \\
        \,G\,=\,K\,(\exp\overline{\mathfrak{a}^{+}})\,K
        \quad&\textnormal{(polar)}.
    \end{cases}
\end{align*}
The Haar measure on $G$ corresponding to the Iwasawa decomposition and the polar decomposition can be described as, for any $f \in C_c(G),$
$$\int_G f(g)\, dg= \int_K \int_{\mathfrak{a}} \int_N f(k (\exp Y) n)\, e^{2 \langle \rho, Y\rangle } dn\,  dY\,dk,$$ and 

$$\int_G f(g)\, dg =\int_K \int_{\overline{\mathfrak{a}^+}} \int_K f(k_1 \exp{Y} k_2) J(\exp Y) dk_1\, dY\, dk_2,$$
respectively. Here the density $J(Y)$ for $Y\in \overline{\mathfrak{a}^+}$ is given by
    \begin{align*}
\textstyle
    J(\exp Y)\,
    =\,c \prod\limits_{\alpha\in\Sigma^{+}}\,
        (\sinh\langle{\alpha,H}\rangle)^{m_{\alpha}}\,
    \asymp\,
        \prod\limits_{\alpha\in\Sigma^{+}}
        \Big\lbrace 
        \frac{\langle\alpha, H\rangle}
        {1+\langle\alpha, H\rangle}
        \Big\rbrace^{m_{\alpha}}\,
        e^{2\langle\rho, H\rangle},
\end{align*} where $c$ is a normalizing constant. Using the polar (Cartan) decomposition we can define  another distance on $X$ called the polyhedral  distance on $X$ defined as $d'(xK, yK):=\langle \rho/ |\rho|, (y^{-1}x)^{+} \rangle$ for all $x, y \in G,$ where $(y^{-1}x)^{+}$ is the $\overline{\mathfrak{a}^+}$-component of $y^{-1}x$ in the polar decomposition. It was proved in \cite{AZ2020} that the Riemannian distance $d$ and the polyhedral distance $d'$ are equivalent.  Any function $f$
 defined on $X$ can be thought of as a function on $G$ which is right $G$-invariant under the action of $K.$ Then it follows that we have a $G$-invariant measure $dx$ on $X$ such that 
 \begin{equation}
     \int_X f(x)\, dx = \int_{K/M} \int_{\mathfrak{a}^+} f(k \exp Y)\, J(\exp Y)\, dY\, dk_M,
 \end{equation} where $dk_M$ is the $K$-invariant measure on $K/M.$
 
 \subsection{Helgason-Fourier transform on Riemannian symmetric spaces} Let $\mathfrak{a}_{\mathbb{C}}^*$ be the complexification of $\mathfrak{a}^*,$ that is, the set of the all complex-valued real linear functionals on $\mathfrak{a}.$ The usual extension of the Killing form $\mathfrak{B}$ on $\mathfrak{a}_{\mathbb{C}}^*$ by duality and conjugate linearity is again denoted by $\mathfrak{B}.$ For a nice function $f$ the Helgason Fourier transform of $\mathcal{H} f$ is a function on $\mathfrak{a}_{\mathbb{C}}^* \times K/M$ defined by 
 \begin{equation}
     (\mathcal{H}f)(\lambda, kM):= \int_X f(x) e^{\langle i\lambda-\rho, H(g^{-1}k)  \rangle}\, dx, \quad \lambda \in \mathfrak{a}_{\mathbb{C}}^*, \,\, kM \in K/M,
 \end{equation} whenever the integral exists. At times, we also denote $\mathcal{H}f$ by $\widehat{f}.$ It is known that the map $f \mapsto \mathcal{H}(f)$ extends to isometry of $L^2(X)$ onto $L^2(\mathfrak{a}_+ \times K, |c(\lambda)|^2 d\lambda dk),$ where $c(\lambda)$ denotes Harish-Chandra's $c$-function. 
 
Let us introduce Harish-Chandra's elementary spherical function in the following form:
\begin{equation}
    \varphi_{\lambda}(g):=\int_{K} e^{-(i\lambda+\rho)(H(g^{-1}k))}dk,\,\,\forall k\in K,\,\,\lambda\in \mathfrak{a}^{*}_{\mathbb{C}}.
\end{equation}

The elementary spherical function $\varphi_{0}$ satisfies the following global estimate \cite[Proposition  2.2.1]{AJ99}:
\begin{equation} \label{gsf}
    \varphi_{0}(\exp H)\asymp \left(\prod_{\alpha\in \sum_{0}^{+}}(1+\langle \alpha, H \rangle)\right)e^{-\rho(H)},\,\,\,\forall H\in \overline{\mathfrak{a}^{+}}.
\end{equation}
Moreover, we have
\begin{equation}
    |\varphi_{\lambda}(g)|\leq \varphi_{0}(g)\leq 1,\,\,\,\,\forall \lambda \in \overline{\mathfrak{a}^{*}_{+}}.
\end{equation}

For $\xi \in [0,\infty), \sigma \in \mathbb{R}$,  let $G_{\xi, \sigma}(x)$ be the Schwartz kernel of the operator $(-\Delta-|\rho|^{2}+\xi^2 )^{-\frac{\sigma}{2}}$ if it exists. We call
 $G_{\xi, \sigma}(x)$ the Bessel-Green-Riesz kernel or simply the Riesz potential. The Riesz potential satisfies the following estimate (see [Theorem 4.2.2, \cite{AJ99}]).
\begin{equation}\label{fund}
\begin{split}
G_{\xi, \sigma}(x)\asymp
\begin{cases}
  \begin{rcases}
   |x|^{\sigma-n},\,\,\,&0<\sigma<n\\
   \log\left(\frac{1}{|x|}\right),\,\,\,& \sigma=n\\
   1, & \sigma>n
  \end{rcases}, \,\,\,\,|x|<1,\\
   |x|^{\frac{\sigma-l-1}{2}-|\Sigma^{+}_0|}\varphi_{0}(x) e^{-\xi|x|},\,\,\xi>0,\,\,\sigma>0,\,\,|x|\geq1.\,\,\,
 \end{cases}
 \end{split}
\end{equation}
Throughout this paper, the symbol $A\asymp B$ means that $\exists\,C_{1},C_{2}>0$ such that $C_{1}A\leq B\leq C_{2}A$.

The Sobolev space on the symmetric space of noncompact type $X$ for $0<\sigma\in\mathbb{R}$ and $1<p<\infty$ is defined as 
\begin{equation} \label{SOS}
    H^{\sigma,p}(X):=\left\{u:u\in L^{p}(X),\,\,(-\Delta)^{\frac{\sigma}{2}}u\in L^{p}(X)\right\},
\end{equation}
endowed with the norm
\begin{equation}\label{Sobnorm}
    \|u\|_{H^{\sigma,p}(X)}:=   \|(-\Delta)^{\frac{\sigma}{2}}u\|_{L^{p}(X)}+\|u\|_{L^{p}(X)}.
\end{equation}

Then it follows from \cite[Theorem 4.4]{Stri83} that
\begin{equation} \label{equiso}
    \|u\|_{H^{\sigma,p}(X)} \approx \|(\xi^2-|\rho|^2-\Delta)^{\frac{\sigma}{2}}u\|_{L^{p}(X)},
\end{equation} where $\xi$ is large enough.

{\it Here after, whenever we deal with $G_{\xi, \sigma}$ we will always assume that $\xi$ is large enough.}

\subsection{Integral Hardy inequalities on metric measure spaces}
Let us  consider metric measure spaces $\mathbb Y$ with a Borel measure $dx$ allowing for the following {\em polar decomposition} at $a\in{\mathbb Y}$: we assume that there is a locally integrable function $J \in L^1_{loc}(\mathbb Y)$  such that for all $f\in L^1(\mathbb Y)$ we have
   \begin{equation}\label{EQ:polarintro}
   \int_{\mathbb Y}f(x)dx= \int_0^{\infty}\int_{\Sigma_r} f(r,\omega) \lambda(r,\omega) \,d\omega_{r} \,dr,
   \end{equation}
    for the  set $\Sigma_r=\{x\in\mathbb{Y}:d(x,a)=r\}\subset \mathbb Y$ with a measure on it denoted by $d\omega_r$, and $(r,\omega)\rightarrow a $ as $r\rightarrow0$. Examples of such metric measure spaces are Euclidean spaces, homogeneous Lie groups, and Riemannian symmetric spaces of noncompact type. We denote $|x|_a:=d(x, a).$
    
    The class of such metric measure spaces was introduced in \cite{RV}, where the following integral Hardy inequality was obtained. 
\begin{theorem} \label{IntHar1}
Let $1<p\le q <\infty$ and let $s>0$. Let $\mathbb Y $ be a metric measure space with a polar decomposition \eqref{EQ:polarintro} at $a$. 
Let $u,v> 0$ be measurable functions  positive a.e in $\mathbb Y$  such that $u\in L^1(\mathbb Y\backslash \{a\})$ and $v^{1-p'}\in L^1_{loc}(\mathbb Y)$. Denote
\begin{align}
U(x):= { \int_{\mathbb Y\backslash{B(a,|x|_a )}} u(y) dy}  \quad
\text{and} 
\quad
V(x):= \int_{B(a,|x|_a  )}v^{1-p'}(y)dy\nonumber. 
\end{align}
Then the inequality
\begin{equation}\label{EQ:Hardy1}
\bigg(\int_\mathbb Y\bigg(\int_{B(a,\vert x \vert_a)}\vert f(y) \vert dy\bigg)^q u(x)dx\bigg)^\frac{1}{q}\le C\bigg\{\int_{\mathbb Y} {\vert f(x) \vert}^pv(x)dx\bigg\}^{\frac1p}
\end{equation}
holds for all measurable functions $f:\X\to{\mathbb C}$ if and only if  any of the following equivalent conditions holds:

\begin{enumerate}
\item $\mathcal D_{1} :=\sup\limits_{x\not=a} \bigg\{U^\frac{1}{q}(x) V^\frac{1}{p'}(x)\bigg\}<\infty.$
\end{enumerate}

\begin{enumerate}\setcounter{enumi}{1}
\item $\mathcal D_{2}:=\sup\limits_{x\not=a} \bigg\{\int_{\mathbb Y\backslash{B(a,|x|_a )}}u(y)V^{q(\frac{1}{p'}-s)}(y)dy\bigg\}^\frac{1}{q}V^s(x)<\infty.$
\item $\mathcal D_{3}:=\sup\limits_{x\not=a}\bigg\{\int_{B(a,|x|_a)}u(y)V^{q(\frac{1}{p'}+s)}(y)dy\bigg\}^{\frac{1}{q}}V^{-s}(x)<\infty $, provided that $u,v^{1-p'}\in L^1(\X)$.
\end{enumerate}

\begin{enumerate}\setcounter{enumi}{3}
\item $\mathcal D_{4}:=\sup\limits_{x\not=a}\bigg\{\int_{B(a,\vert x \vert_a)}v^{1-p'}(y) U^{p'(\frac{1}{q}-s)}(y)dy\bigg\}^\frac{1}{p'}U^s(x)<\infty.$ 

\item $\mathcal D_{5}:=\sup\limits_{x\not=a}\bigg\{\int_{\mathbb Y\backslash{B(a,\vert x \vert_a )}}v^{1-p'}(y)U^{p'(\frac{1}{q}+s)}(y)dy\bigg\}^\frac{1}{p'}U^{-s}(x)<\infty$, provided that $u,v^{1-p'}\in L^1(\X)$.
\end{enumerate}
Moreover, the constant $C$ for which \eqref{EQ:Hardy1} holds and quantities $\mathcal D_{1}-\mathcal D_{5}$ are related by 
\begin{equation}\label{EQ:constants}
\mathcal D_{1} \leq C\leq \mathcal D_1(p')^{\frac{1}{p'}} p^\frac{1}{q},
\end{equation}   
and 
$$\mathcal D_1 \le \left(\max\left\{1,{p'}{s}\right\}\right)^\frac{1}{q}\mathcal D_2, \;\mathcal D_2 \le \left(\max\left\{1,\frac{1}{p's}\right\}\right)^\frac{1}{q} \mathcal D_1,$$ 
$$\left(\frac{sp'}{1+p's}\right)^\frac{1}{q} \mathcal D_3 \le \mathcal D_1\le(1+sp')^\frac{1}{q}\mathcal D_3,$$
$$\mathcal D_1 \le (\max\left\{1,qs\right\})^\frac{1}{p'} \mathcal D_4,\; \mathcal D_4 \le \left(\max\left\{1,\frac{1}{qs}\right\}\right)^\frac{1}{p'}
\mathcal D_1,$$ 
$$\left(\frac{sq}{1+qs}\right)^\frac{1}{p'}\mathcal D_5 \le \mathcal D_1 \le (1+sq)^\frac{1}{p'} \mathcal D_5.$$
\end{theorem}
Similarly, in \cite{RV}, the authors obtained the adjoint integral Hardy inequality in the following form:
\begin{theorem}\label{IntHar2}
Let $1<p\le q <\infty$ and let $s>0$. Let $\mathbb Y $ be a metric measure space with a polar decomposition \eqref{EQ:polarintro} at $a$. 
Let $u,v> 0$ be measurable functions  positive a.e in $\mathbb Y$  such that $u\in L^1(\mathbb Y\backslash \{a\})$ and $v^{1-p'}\in L^1_{loc}(\mathbb X)$. Denote
\begin{align}
U(x):= { \int_{{B(a,|x|_a )}} u(y) dy}\quad  \text{and} \quad
V(x):= \int_{\mathbb Y\backslash B(a,|x|_a  )}v^{1-p'}(y)dy\nonumber. 
\end{align}
Then the inequality
\begin{equation}\label{EQ:Hardy2}
\bigg(\int_\mathbb Y\bigg(\int_{\X\setminus B(a,\vert x \vert_a)}\vert f(y) \vert dy\bigg)^q u(x)dx\bigg)^\frac{1}{q}\le C\bigg\{\int_{\mathbb Y} {\vert f(x) \vert}^pv(x)dx\bigg\}^{\frac1p}
\end{equation}
holds for all measurable functions $f:\X\to{\mathbb C}$ if and only if  any of the following equivalent conditions holds:
\begin{enumerate}
\item $\mathcal D^{*}_{1} :=\sup\limits_{x\not=a} \bigg\{U^\frac{1}{q}(x) V^\frac{1}{p'}(x)\bigg\}<\infty.$
\end{enumerate}

\begin{enumerate}\setcounter{enumi}{1}
\item $\mathcal D^{*}_{2}:=\sup\limits_{x\not=a} \bigg\{\int_{\X \backslash{B(a,|x|_a )}}u(y)V^{q(\frac{1}{p'}-s)}(y)dy\bigg\}^\frac{1}{q}V^s(x)<\infty.$
\item $\mathcal D^{*}_{3}:=\sup\limits_{x\not=a}\bigg\{\int_{\X\setminus B(a,|x|_a)}u(y)V^{q(\frac{1}{p'}+s)}(y)dy\bigg\}^{\frac{1}{q}}V^{-s}(x)<\infty $, provided that $u,v^{1-p'}\in L^1(\X)$.
\end{enumerate}

\begin{enumerate}\setcounter{enumi}{3}
\item $\mathcal D^{*}_{4}:=\sup\limits_{x\not=a}\bigg\{\int_{\X\setminus B(a,\vert x \vert_a)}v^{1-p'}(y) U^{p'(\frac{1}{q}-s)}(y)dy\bigg\}^\frac{1}{p'}U^s(x)<\infty.$ 

\item $\mathcal D^{*}_{5}:=\sup\limits_{x\not=a}\bigg\{\int_{{B(a,\vert x \vert_a )}}v^{1-p'}(y)U^{p'(\frac{1}{q}+s)}(y)dy\bigg\}^\frac{1}{p'}U^{-s}(x)<\infty$, provided that $u,v^{1-p'}\in L^1(\X)$.
\end{enumerate}
\end{theorem}
\begin{rem}
In this paper, for using integral Hardy inequalities, we use conditions $\mathcal D_{1}$ and $\mathcal D^{*}_{1}$ in the previous theorems. The class of general metric measure spaces having a polar decomposition was analysed in \cite{ZhRuz}.
\end{rem}

\section{Stein-Weiss and Hardy-Littlewood-Sobolev inequalities  on symmetric spaces of noncompact type} \label{sec3}
Let us show the Stein-Weiss  inequality on symmetric space of noncompact type. The proof is an adaption of the argument in \cite{RY} that was developed for the graded Lie groups,  however here, it depends on the induced geometry of the space in a more substantial way.
\begin{theorem}\label{thmSteinWeiss}
Let $X$ be a symmetric space of noncompact type of dimension $n\geq 3$ and  rank $l\geq1$.  Let $0<\sigma< n$, $1<p<\infty$,  $\alpha<\frac{n}{p'}$, $\beta<\frac{n}{q},$ $\alpha +\beta\geq0$ and $\frac{\sigma-\alpha -\beta}{n}=\frac{1}{p}-\frac{1}{q}$.  Then, for sufficiently large $\xi>0$ and $1<p\leq q<\infty,$ we have
\begin{equation} \label{SWineq}
    \left(\int_{X}\left|\int_{X}G_{\xi, \sigma}(y^{-1}x)u(y)dy\right|^{q}\frac{dx}{|x|^{\beta q}}\right)^{\frac{1}{q}}\leq C\||x|^{\alpha}u\|_{L^{p}(X)},
\end{equation}
where $C$ is a positive constant independent of $u$.
\end{theorem}
\begin{proof}
Let us begin the proof with the quantity on the left hand side of \eqref{SWineq} and divide it into three part as follows:
\begin{equation}
    \begin{split}
        \int_{X}\left|\int_{X}G_{\xi, \sigma}(y^{-1}x)u(y)dy\right|^{q}|x|^{-\beta q}dx\leq C\left(I^{q}_{1}+I^{q}_{2}+I^{q}_{3}\right),
    \end{split}
\end{equation}
where 
$$I^{q}_{1}=\int_{ X}\left(\int_{B\left(0,\frac{|x|}{2}\right)}|G_{\xi, \sigma}(y^{-1}x)u(y)|\,dy\right)^{q}|x|^{-\beta q}dx,$$
$$I^{q}_{2}=\int_{ X}\left(\int_{B\left(0,2|x|\right)\setminus B\left(0,\frac{|x|}{2}\right)} |G_{\xi, \sigma}(y^{-1}x)u(y)|\,dy\right)^{q}|x|^{-\beta q}dx,$$
$$I^{q}_{3}=\int_{X}\left(\int_{ X\setminus B\left(0,2|x|\right)}|G_{\xi, \sigma}(y^{-1}x)u(y)|\, dy\right)^{q}|x|^{-\beta q}dx.$$

Now, we  will estimate $I_i^q, i=1,2,3.$ This will be completed in three different steps. 

\textbf{Step 1.} In this step we consider $I_{1}^q$.
By using the triangle inequality for the Riemannian distance with $2|y|\leq |x|$, we obtain
\begin{equation*}
    |x|\leq |y^{-1}x|+|y|\leq |y^{-1}x|+ \frac{|x|}{2},
\end{equation*}
which  implies that 
$$\frac{|x|}{2}\leq |y^{-1}x|.$$

Let us now consider the two different cases. 

{\bf Case (a).} When $\frac{|x|}{2} \geq 1.$
Let us first find an estimate of $G_{\xi, \sigma}(y^{-1}x)$ in terms of $|x|,$ using the fact $\frac{|x|}{2}\leq |y^{-1}x|,$ in the range when $2|y| \leq |x|.$ Note that by \eqref{fund}, we have
$$G_{\xi, \sigma}(y^{-1}x) \asymp |y^{-1}x|^{\frac{\sigma-l-1}{2}-|\Sigma_0^+|} \varphi_0(y^{-1}x) e^{-\xi |y^{-1}x|}.$$ Since the Riemannian distance and the polyhedral  distance are equivalent and $G_{\xi, \sigma}$ is $K$-biinvariant, we have 
$$G_{\xi, \sigma}(y^{-1}x):=G(\exp(y^{-1}x)^+) \asymp |(y^{-1}x)^+|^{\frac{\sigma-l-1}{2}-|\Sigma_0^+|} \varphi_0(\exp(y^{-1}x)^+) e^{-\xi |(y^{-1}x)^+|}.$$
The estimate \eqref{gsf} of the ground spherical function $\varphi_0(\exp(y^{-1}x)^+) \lesssim |(y^{-1}x)^+|^{|\Sigma_0^+|} e^{- \langle \rho, (y^{-1}x)^+ \rangle}$ yields,
$$G_{\xi, \sigma}(y^{-1}x) \lesssim |(y^{-1}x)^+|^{\frac{\sigma-l-1}{2}-|\Sigma_0^+|} |(y^{-1}x)^+|^{|\Sigma_0^+|} e^{- \langle \rho,\, (y^{-1}x)^+ \rangle} e^{-\xi |(y^{-1}x)^+|}.$$
Now, by  $ \langle \rho, x^+ \rangle \leq \langle \rho,  y^+ \rangle+\langle \rho,  (y^{-1}x)^+ \rangle,$ $ \frac{|x^+|}{2} \leq |(y^{-1}x)^{+}| \leq \frac{3|x^+|}{2}$ using $|y^+| \leq \frac{|x^+|}{2},$ and Cauchy-Schwarz inequality,  we get
\begin{align*}
    G_{\xi, \sigma}(y^{-1}x) \lesssim & |x^+|^{\frac{\sigma-l-1}{2}} \,e^{- \langle \rho,\, x^+ \rangle} e^{\langle \rho, y^+ \rangle}  e^{-\frac{\xi}{2}|x^+|} 
    \\ \lesssim & |x^+|^{\frac{\sigma-l-1}{2}} \,e^{- \langle \rho,\, x^+ \rangle} e^{|\rho||y^+| } e^{-\frac{\xi}{2}|x^+|} \lesssim |x^+|^{\frac{\sigma-l-1}{2}} \,e^{- \langle \rho,\, x^+ \rangle} e^{|\rho|\frac{|x^+|}{2} } e^{-\frac{\xi}{2}|x^+|}.
\end{align*}

Therefore, we obtain
\begin{equation}
    \begin{split}
        I^{q}_{1}&=\int_{ X}\left(\int_{B\left(0,\frac{|x|}{2}\right)}|G_{\xi, \sigma}(y^{-1}x)u(y)|\,dy\right)^{q}|x|^{-\beta q}dx\\&
        \leq C \int_{ X}\left(\int_{B\left(0,\frac{|x|}{2}\right)}u(y)dy\right)^{q} \left( |x^+|^{\frac{\sigma-l-1}{2}} \,e^{- \langle \rho,\, x^+ \rangle} e^{|\rho|\frac{|x^+|}{2} } e^{-\frac{\xi}{2}|x^+|} \right)^q |x|^{-\beta q}dx.
    \end{split}
\end{equation}
Our aim is to show the following inequality,
\begin{align}
         &\left(\int_{ X}\left(\int_{B\left(0,\frac{|x|}{2}\right)}u(y)dy\right)^{q}  \left( |x^+|^{\frac{\sigma-l-1}{2}} \,e^{- \langle \rho,\, x^+ \rangle} e^{|\rho|\frac{|x^+|}{2} } e^{-\frac{\xi}{2}|x^+|} \right)^q|x|^{-\beta q}dx\right)^{\frac{1}{q}} \\&\leq \nonumber C\left(\int_{X}|y|^{\alpha p}|u(y)|^{p}dy\right)^\frac{1}{p},
\end{align}
 and, for this purpose, we need to check the condition $\mathcal D_{1}$ in Theorem \ref{IntHar1} which turns out to be the following,
\begin{align} \label{D13.5}
 \nonumber   \mathcal{D}_{1}&=\sup_{x \neq 0}\left(\int_{X \backslash B\left(0,\frac{|x|}{2}\right)} |z^+|^{q\frac{\sigma-l-1}{2}} \,e^{-q\langle \rho,\, z^+ \rangle} e^{q|\rho|\frac{|z^+|}{2} } e^{-\frac{\xi}{2}|z^+|q} |z|^{-\beta q}dx\right)^{\frac{1}{q}}\left(\int_{B\left(0,\frac{|x|}{2}\right)}|z|^{\alpha p(1-p')}dz\right)^{\frac{1}{p'}} \\& =\sup_{R > 0}\left(\int_{|z| \geq R} |z^+|^{q\frac{\sigma-l-1}{2}} \,e^{-q\langle \rho,\, z^+ \rangle} e^{q|\rho|\frac{|z^+|}{2} } e^{-\frac{\xi}{2}|z^+|q} |z|^{-\beta q}dx\right)^{\frac{1}{q}}\left(\int_{|z| \leq R}|z|^{\alpha p(1-p')}dz\right)^{\frac{1}{p'}},
\end{align}
by setting $|x|/2=R \geq 1.$

 By using the Cartan decomposition with $R \geq 1$ in this case we get
\begin{equation}\label{estlar1}
    \begin{split}
        &\int_{|z|\geq R}|z^+|^{\frac{\sigma-l-1}{2}q} \,e^{-q\langle \rho,\, z^+ \rangle} e^{q|\rho|\frac{|z^+|}{2} } e^{-\frac{\xi}{2}|z^+|q} |z|^{-\beta q}dz \\&
        \leq C\int_{\{H\in \overline{\mathfrak{a}^{+}}:|H|\geq R\}}|H|^{q\frac{\sigma-l-1}{2}-\beta q} \,e^{-q\langle \rho,\, H \rangle} e^{q|\rho|\frac{|H|}{2} } e^{-\frac{\xi}{2}|H|q} e^{2 \langle \rho, H \rangle} dH\\&
        \leq C e^{-\xi q\frac{R}{4}} \int_{\{H\in \mathfrak{a}:|H|\geq R\}}|H|^{q\frac{\sigma-l-1}{2}-\beta q} e^{-\left(\frac{\xi}{4}-\frac{|\rho|}{2} \right)|H|q} \,e^{-(q-2)\langle \rho,\, H \rangle}  dH\\&
        \leq C e^{-\xi q\frac{R}{4}},
    \end{split}
\end{equation} when $q\geq 2$ and $\xi$ large enough (e.g. $\xi \geq  2|\rho|$). For $q<2,$ with same calculations as in \eqref{estlar1} we get 
\begin{equation}\label{estlar11}
    \begin{split}
        &\int_{|z|\geq R}|z^+|^{\frac{\sigma-l-1}{2}q} \,e^{-q\langle \rho,\, z^+ \rangle} e^{q|\rho|\frac{|z^+|}{2} } e^{-\frac{\xi}{2}|z^+|q} |z|^{-\beta q}dz \\&
        \leq C e^{-\xi q\frac{R}{4}} \int_{\{H\in \mathfrak{a}:|H|\geq R\}}|H|^{q\frac{\sigma-l-1}{2}-\beta q} e^{-\left(\frac{\xi}{4}-\frac{|\rho|}{2} \right)|H|q} \,e^{(2-q)\langle \rho,\, H \rangle}  dH\\& \leq C e^{-\xi q\frac{R}{4}} \int_{\{H\in \mathfrak{a}:|H|\geq R\}}|H|^{q\frac{\sigma-l-1}{2}-\beta q} e^{-\left(\frac{\xi}{4}-\frac{|\rho|}{2} \right)|H|q+(2-q) |\rho| |H|}  dH\\&
         \leq C e^{-\xi q\frac{R}{4}} \int_{\{H\in \mathfrak{a}:|H|\geq R\}}|H|^{q\frac{\sigma-l-1}{2}-\beta q} e^{-\left(\frac{\xi}{4}+\frac{q}{2}|\rho|-2|\rho| \right) |H|}  dH \leq C e^{-\xi q\frac{R}{4}},
    \end{split}
    \end{equation}
whenever $\frac{\xi}{4}+\frac{q}{2}|\rho|-2|\rho| \geq 0,$ which holds as $\xi$ is large enough.

Let us consider the second integral,
\begin{equation}\label{estlar2}
    \begin{split}
        \int_{|z|\leq R}|z|^{\alpha p(1-p')}dx&\stackrel{ p(1-p')=-p'}=\int_{0<|z|<1}|z|^{-\alpha p'}dz+\int_{1\leq |z|\leq R}|z|^{- \alpha p'}dz\\&
        \leq  C\int_{\{H\in \overline{\mathfrak{a}^{+}}:|H|< 1\}}|H|^{-\alpha p'}|H|^{n-l}dH+C \int_{\{H\in \overline{\mathfrak{a}^{+}}:1\leq |H|\leq R\}}|H|^{-\alpha p'}e^{2|\rho| |H|}dH\\&
        \leq C\int_{0}^{1}r^{-\alpha p'}r^{n-l}r^{l-1}dr+Ce^{2|\rho|R}\int_{1}^{R}r^{-\alpha p'+l-1}dr\\&
        \stackrel{\alpha<\frac{n}{p'}}= C(1+e^{2|\rho|R}).
    \end{split}
\end{equation}
By combining the estimates \eqref{estlar1}, \eqref{estlar11} and  \eqref{estlar2} with the fact that $\xi$ is sufficiently large (e.g., $\xi \geq 8|\rho|$) and substituting back in \eqref{D13.5},  we deduce that
\begin{align*}
    \mathcal{D}_{1}&=\sup_{R>0}\left(\int_{|z|\geq R}|z^+|^{\frac{\sigma-l-1}{2}q} \,e^{-q\langle \rho,\, z^+ \rangle} e^{q|\rho|\frac{|z^+|}{2} } e^{-\frac{\xi}{2}|z^+|q} |z|^{-\beta q} dz \right)^{\frac{1}{q}}&\left(\int_{|z|\leq R}|z|^{\alpha p(1-p')}dz\right)^{\frac{1}{p'}}\\&\leq C\sup_{R>0} e^{-\xi \frac{R}{4}}(1+e^{2|\rho|R})^{\frac{1}{p'}}\leq  C\sup_{R>0} e^{-\xi \frac{R}{4}}(1+e^{2|\rho|R}) < +\infty,
\end{align*} as $p'>1$ and $\xi$ is sufficiently large.

{\bf Case (b).} When $0<\frac{|x|}{2} < 1.$ Again, we need to consider the integral 
$$ I^{q}_{1}=\int_{ X}\left(\int_{B\left(0,\frac{|x|}{2}\right)}|G_{\xi, \sigma}(y^{-1}x)u(y)|\,dy\right)^{q}|x|^{-\beta q}dx.$$
Since $\frac{|x|}{2} \leq |y^{-1}x|,$ we can divide the study of this  integral into two parts: one when $|y^{-1}x| < 1$ and the second when $|y^{-1}x| \geq 1.$ In fact, the case when $|y^{-1}x| \geq 1,$ is similar to the Case (a). So by proceeding similar to case (a) with the obtained estimate of $G_{\xi, \sigma}(y^{-1}x)$ when $|y^{-1}x| \geq 1,$ we have 

\begin{equation}
    \begin{split}
        I^{q}_{1}&=\int_{ X}\left(\int_{B\left(0,\frac{|x|}{2}\right)}|G_{\xi, \sigma}(y^{-1}x)u(y)|\,dy\right)^{q}|x|^{-\beta q}dx\\&
        \leq C \int_{ X}\left(\int_{B\left(0,\frac{|x|}{2}\right)}u(y)dy\right)^{q} \left( |x^+|^{\frac{\sigma-l-1}{2}} \,e^{- \langle \rho,\, x^+ \rangle} e^{|\rho|\frac{|x^+|}{2} } e^{-\frac{\xi}{2}|x^+|} \right)^q |x|^{-\beta q}dx.
    \end{split}
\end{equation}
To establish the inequality 
\begin{align}
         &\left(\int_{ X}\left(\int_{B\left(0,\frac{|x|}{2}\right)}u(y)dy\right)^{q}  \left( |x^+|^{\frac{\sigma-l-1}{2}} \,e^{- \langle \rho,\, x^+ \rangle} e^{|\rho|\frac{|x^+|}{2} } e^{-\frac{\xi}{2}|x^+|} \right)^q|x|^{-\beta q}dx\right)^{\frac{1}{q}} \\&\leq \nonumber C\left(\int_{X}|y|^{\alpha p}|u(y)|^{p}dy\right)^\frac{1}{p},
\end{align}
we need to verify that the following holds,
\begin{align}
    \mathcal{D}_{1} =\sup_{R > 0}\left(\int_{|z| \geq R} |z^+|^{\frac{\sigma-l-1}{2}q}\,e^{-q\langle \rho,\, z^+ \rangle} e^{q|\rho|\frac{|z^+|}{2} } e^{-\frac{\xi}{2}|z^+|q} |z|^{-\beta q}dx\right)^{\frac{1}{q}}\left(\int_{|z| \leq R}|z|^{\alpha p(1-p')}dz\right)^{\frac{1}{p'}}<\infty.
\end{align}
Here $0<R=|x|/2 < 1.$

In order to check that $\mathcal{D}_1<\infty,$  using the  Cartan decomposition  we get
\begin{equation}\label{estlar1v}
    \begin{split}
        &\int_{|z|\geq R}|z^+|^{\frac{\sigma-l-1}{2}q} \,e^{-q\langle \rho,\, z^+ \rangle} e^{q|\rho|\frac{|z^+|}{2} } e^{-\frac{\xi}{2}|z^+|q} |z|^{-\beta q}dz \\& =\int_{R \leq |z| < 1} |z^+|^{\frac{\sigma-l-1}{2}q} \,e^{-q\langle \rho,\, z^+ \rangle} e^{q|\rho|\frac{|z^+|}{2} } e^{-\frac{\xi}{2}|z^+|q} |z|^{-\beta q}dz \\& \quad \quad \quad \quad\quad+ \int_{|z| \geq 1} |z^+|^{\frac{\sigma-l-1}{2}q} \,e^{-q\langle \rho,\, z^+ \rangle} e^{q|\rho|\frac{|z^+|}{2} } e^{-\frac{\xi}{2}|z^+|q} |z|^{-\beta q}dz\\&
        \leq C \Big\{ \int_{\{H\in \overline{\mathfrak{a}^{+}}:R \leq |H|< 1\}} |H|^{\frac{\sigma-l-1}{2}q-\beta q} \,e^{-q\langle \rho,\, H \rangle} e^{q|\rho|\frac{|H|}{2} } e^{-\frac{\xi}{2}|H|q} e^{2 \langle \rho, H \rangle} dH \\& \quad\quad\quad +\int_{\{H\in \overline{\mathfrak{a}^{+}}:|H|\geq 1\}}|H|^{\frac{\sigma-l-1}{2}q-\beta q} \,e^{-q\langle \rho,\, H \rangle} e^{q|\rho|\frac{|H|}{2} } e^{-\frac{\xi}{2}|H|q} e^{2 \langle \rho, H \rangle} dH \Big\} =C\{J_1+J_2\}.
    \end{split}
\end{equation}
It is easy to see that $J_1<\infty$ is an integral of a continuous function over a compact set. To see that $J_2<\infty$ one can argue verbatim as in  \eqref{estlar1} and \eqref{estlar11} by considering two case $q\geq 2$ and $q<2$ with $\xi$ sufficiently large (in this case, $\xi>2|\rho|$ will work).  Therefore, we get 
\begin{equation} \label{D11}
    \int_{|z|\geq R}|z^+|^{\frac{\sigma-l-1}{2}q} \,e^{-q\langle \rho,\, z^+ \rangle} e^{q|\rho|\frac{|z^+|}{2} } e^{-\frac{\xi}{2}|z^+|q} |z|^{-\beta q}dz \leq C.
\end{equation}
Also, we compute the following integral for $0<R<1$ with $\alpha<\frac{n}{p'}$:
\begin{equation}\label{D12}
    \begin{split}
      \int_{|z|\leq R}|z|^{\alpha p(1-p')}dz&\stackrel{ p(1-p')=-p'}\leq C \int_{0}^{R}r^{-\alpha p'+n-1}dr\\&
      =CR^{-\alpha p'+n} \leq C<\infty.
    \end{split}
\end{equation}
Thus, from \eqref{D11} and \eqref{D12} we have 
\begin{equation*}
\begin{split}
      \mathcal{D}_{1} =\sup_{R > 0}\left(\int_{|z| \geq R} |z^+|^{\frac{\sigma-l-1}{2}q} \,e^{-q\langle \rho,\, z^+ \rangle} e^{q|\rho|\frac{|z^+|}{2} } e^{-\frac{\xi}{2}|z^+|q} |z|^{-\beta q}dx\right)^{\frac{1}{q}}\left(\int_{|z| \leq R}|z|^{\alpha p(1-p')}dz\right)^{\frac{1}{p'}}\leq C<\infty.
\end{split}
\end{equation*}

Next, we consider the remaining case, when $|y^{-1}x|<1.$ In this case, using \eqref{fund} we have $G_{\xi, \sigma}(y^{-1}x) \asymp |y^{-1}x|^{\sigma-n}$ for $0<\sigma<n.$ So, by using $\frac{|x|}{2} \leq |y^{-1}x|$ we get $G_{\xi, \sigma}(y^{-1}x) \asymp |y^{-1}x|^{\sigma-n} \leq \frac{|x|^{\sigma-n}}{2} \asymp G_{\xi, \sigma}(\frac{x}{2})$ and therefore,
\begin{equation}
    \begin{split}
        I^{q}_{1}=&\int_{ X}\left(\int_{B\left(0,\frac{|x|}{2}\right)}|G_{\xi, \sigma}(y^{-1}x)u(y)|\,dy\right)^{q}|x|^{-\beta q}dx\\&
        \quad\quad\quad\quad\leq C \int_{ X}\left(\int_{B\left(0,\frac{|x|}{2}\right)}u(y)dy\right)^{q} \left( G_{\xi, \sigma}\left(\frac{x}{2}\right) \right)^{q} |x|^{-\beta q}dx.
    \end{split}
\end{equation}
To show  the inequality 
\begin{align}
         &\left(\int_{ X}\left(\int_{B\left(0,\frac{|x|}{2}\right)}u(y)dy\right)^{q}  \left( G_{\xi, \sigma}\left(\frac{x}{2}\right) \right)^{q}|x|^{-\beta q}dx\right)^{\frac{1}{q}}\leq \nonumber C\left(\int_{X}|y|^{\alpha p}|u(y)|^{p}dy\right)^\frac{1}{p},
\end{align}
it is enough to show that,
\begin{align}
    \mathcal{D}_{1} =\sup_{R > 0}\left(\int_{|z| \geq R} \left( G\left(\frac{z}{2}\right) \right)^{q}  |z|^{-\beta q}dz\right)^{\frac{1}{q}}\left(\int_{|z| \leq R}|z|^{\alpha p(1-p')}dz\right)^{\frac{1}{p'}}<\infty.
\end{align}
Here $0<R=|x|/2 < 1.$
Then we have
\begin{equation}\label{D11v}
    \begin{split}
     \int_{|z| \geq R} \left( G_{\xi, \sigma}\left(\frac{z}{2}\right) \right)^{q} |z|^{-\beta q}dz&=\int_{R\leq |z|\leq 1}\left( G_{\xi, \sigma}\left(\frac{z}{2}\right) \right)^{q}|x|^{-\beta q}dz+\int_{|z|\geq 1} \left( G_{\xi, \sigma}\left(\frac{z}{2}\right) \right)^{q}|z|^{-\beta q}dz \\& \asymp \int_{R\leq |z|\leq 1}|z|^{q(\sigma-n)}|z|^{-\beta q}dz+\int_{|z|\geq 1} \left( |z|^{\frac{\sigma-l-1}{2}-|\Sigma^{+}_0|}\varphi_{0}(z) e^{-\xi|z|} \right)^{q}|z|^{-\beta q}dz\\&
     \leq C\int_{\{H\in \overline{\mathfrak{a}^{+}}:R\leq |H|\leq 1\}}|H|^{(\sigma-n)q-\beta q}|H|^{n-l}dH\\&
     +C\int_{\{H\in \overline{\mathfrak{a}^{+}}: |H|\geq 1\}}|H|^{\frac{\sigma-l-1}{2}q-|\Sigma^{+}|q-\beta q}e^{-\xi q\frac{|H|}{2}}e^{2|\rho||H|}dH\\&
     \stackrel{\eqref{estlar1}}\leq C\int_{2R}^{2}r^{(\sigma-n)q-\beta q+n-1}dr+Ce^{-\xi \frac{R}{2}}\\&
     \leq C(R^{(\sigma-n)q-\beta q+n}+e^{-\xi \frac{R}{2}}).
    \end{split}
\end{equation}
Also, we compute the following integral for $0<R<1$ with $\alpha<\frac{n}{p'}$:
\begin{equation}\label{D12v}
    \begin{split}
      \int_{|z|\leq R}|z|^{\alpha p(1-p')}dz&\stackrel{\eqref{estlar2}}\leq C \int_{0}^{2R}r^{-\alpha p'+n-1}dr\\&
      =CR^{-\alpha p'+n}.
    \end{split}
\end{equation}
Thus, from \eqref{D11v}, \eqref{D12v} and $\frac{\sigma-\alpha -\beta}{n}=\frac{1}{p}-\frac{1}{q}$, we have 
\begin{equation*}
    \begin{split}
        \left(\int_{|z|\geq R}G_{\xi, \sigma}\left(\frac{z}{2}\right)^q|z|^{-\beta q}dz\right)^{\frac{1}{q}}\left(\int_{|z|\leq R}|z|^{\alpha p(1-p')}dz\right)^{\frac{1}{p'}} 
        &\leq C \left(R^{\frac{(\sigma-n)q-\beta q+n}{q}+\frac{{-\alpha p'+n}}{p'}}+R^{\frac{{-\alpha p'+n}}{p'}}e^{-\xi \frac{R}{2}}\right)\\&
        = C\left(1+R^{\frac{{-\alpha p'+n}}{p'}}e^{-\xi \frac{R}{2}}\right).
    \end{split}
\end{equation*}
Therefore, in this case we get 
\begin{equation*}
\begin{split}
      \mathcal{D}_{1}=\sup_{R>0}\left(\int_{|z|\geq R}G_{\xi, \sigma}\left(\frac{z}{2}\right)^q|z|^{-\beta q}dz\right)^{\frac{1}{q}}\left(\int_{|z|\leq R}|z|^{\alpha p(1-p')}dz\right)^{\frac{1}{p'}} &\leq   C \sup_{R>0}\left(1+R^{\frac{{-\alpha p'+n}}{p'}}e^{-\xi \frac{R}{2}}\right)\\&
      <\infty.
\end{split}
\end{equation*}
The application of Theorem \ref{IntHar1}  completes the proof of case (b).

\textbf{Step 2.} In this step, we consider $I_{3}$. By using the triangle inequality for the Riemannian distance with $|y|\geq 2|x|$, we get
\begin{equation}
    |y|\leq |y^{-1}x|+|x|\leq |y^{-1}x|+\frac{|y|}{2},
\end{equation}
therefore, we have $|x| \leq \frac{|y|}{2}\leq |y^{-1}x|$.
By arguing exactly in the same way as in Case (a) of Step 1 we estimate that 
$$G_{\xi, \sigma}(y^{-1}x) \lesssim |y|^{\frac{\sigma-l-1}{2}} e^{-\langle \rho, y\rangle} e^{|\rho|\frac{|y|}{2}} e^{-\frac{\xi}{2}|y|}.$$ 
Thus, we obtain
\begin{equation} \label{I3q}
    \begin{split}
        I_{3}^{q}&=\int_{X}\left(\int_{ X\setminus B\left(0,2|x|\right)}|G_{\xi, \sigma}(y^{-1}x)u(y)|\,dy\right)^{q}|x|^{-\beta q}dx\\&
        \leq \int_{X}\left(\int_{ X\setminus B\left(0,2|x|\right)}\left(|y|^{\frac{\sigma-l-1}{2}} e^{-\langle \rho, y\rangle} e^{|\rho|\frac{|y|}{2}} e^{-\frac{\xi}{2}|y|}\right)u(y)dy\right)^{q}|x|^{-\beta q}dx.
    \end{split}
\end{equation}
If we show the following condition,
\begin{align}\label{D2}
   \nonumber \mathcal{D}^*_{1}&=\sup_{x \neq 0}\left(\int_{B(0, 2|x|)}|z|^{-\beta q}dz\right)^{\frac{1}{q}}\left(\int_{X \backslash B(0,2|x|)}\left(|z|^{\frac{\sigma-l-1}{2}} e^{-\langle \rho, z\rangle} e^{|\rho|\frac{|z|}{2}} e^{-\frac{\xi}{2}|z|}\right)^{p'}|z|^{-\alpha p'}dz\right)^{\frac{1}{p'}}\\&=\sup_{R > 0}\left(\int_{|z|\leq R}|z|^{-\beta q}dz\right)^{\frac{1}{q}}\left(\int_{|z|\geq R}\left(|z|^{\frac{\sigma-l-1}{2}} e^{-\langle \rho, z\rangle} e^{|\rho|\frac{|z|}{2}} e^{-\frac{\xi}{2}|z|}\right)^{p'}|z|^{-\alpha p'}dz\right)^{\frac{1}{p'}},
\end{align}
then by using the conjugate integral Hardy inequality (see Theorem \ref{IntHar2}), we get
\begin{equation} \label{I3q+}
    \begin{split}
        \int_{X}\left(\int_{ X\setminus B\left(0,2|x|\right)}\left(|y|^{\frac{\sigma-l-1}{2}} e^{-\langle \rho, y\rangle} e^{|\rho|\frac{|y|}{2}} e^{-\frac{\xi}{2}|y|}\right)u(y)dy\right)^{q}|x|^{-\beta q}dx \leq C\left(\int_{X}|x|^{\alpha p} |u(x)|^{p}dx\right)^{\frac{q}{p}}.
    \end{split}
\end{equation}
Let us check the condition \eqref{D2}. Similarly to the previous step, we consider two cases $0<R<1$ and $R
\geq 1$.
Firstly, let us consider $R\geq 1$ and by using the Cartan decomposition, we get, same as \eqref{estlar1} and \eqref{estlar11}, that 
\begin{equation}\label{D22}
    \begin{split}
       \int_{|z|\geq R}\left(|z|^{\frac{\sigma-l-1}{2}} e^{-\langle \rho, z\rangle} e^{|\rho|\frac{|z|}{2}} e^{-\frac{\xi}{2}|z|}\right)^{p'}|z|^{-\alpha p'}dz
       \leq Ce^{-\xi p'\frac{R}{2}},
    \end{split}
\end{equation}
for sufficiently large $\xi$. Then let us compute the first integral in \eqref{D2}, we get
\begin{equation}\label{D21}
    \begin{split}
        \int_{|z|\leq R}|z|^{-\beta q}dz&=\int_{|z|\leq 1}|z|^{-\beta q}dz+\int_{1<|z|\leq R}|z|^{-\beta q}dz\\&
        \leq C \int_{0}^{1}r^{-\beta q+n-1}dr+C\max\{1,R^{-\beta q}\}\int_{1<|x|\leq R}dx\\&
        \stackrel{\beta<\frac{n}{q}} \leq C+CR^{-\beta q+n}e^{2|\rho|R}\\&
        \leq C(1+R^{-\beta q+n}e^{2|\rho|R}).
    \end{split}
\end{equation}
By combining \eqref{D21} and \eqref{D22}, we get
\begin{align*}
    \mathcal{D}_1^*=\sup_{R > 0}\left(\int_{|z|\leq R}|z|^{-\beta q}dz\right)^{\frac{1}{q}}&\left(\int_{|z|\geq R}\left(|z|^{\frac{\sigma-l-1}{2}} e^{-\langle \rho, z\rangle} e^{|\rho|\frac{|z|}{2}} e^{-\frac{\xi}{2}|z|}\right)^{p'}|z|^{-\alpha p'}dz\right)^{\frac{1}{p'}}\\&\leq \sup_{R > 0} Ce^{-\xi \frac{R}{2}}(1+R^{\frac{-\beta q+n}{q}}e^{\frac{2|\rho|R}{q}})<\infty,
\end{align*}
thanks to  sufficiently large $\xi$.

Let us consider the case $0<R<1$. Similarly to \eqref{D21}, for $0<R<1$ we get
\begin{equation} \label{E11}
   \begin{split}
       \int_{|z|\leq R}|z|^{-\beta q}dz&\leq  C \int_{\{H\in \overline{\mathbb{a}^{+}}:|H|\leq R\}}|H|^{-\beta q}|H|^{n-l}dH\\&
       \leq C\int_{0}^{R}r^{-\beta q}r^{n-1}dr\\&
       \stackrel{\beta<\frac{n}{q}}=CR^{-\beta q+n} \leq C.
   \end{split} 
\end{equation}
Next we will estimate the second integral of \eqref{D2} for $0<R<1$. Thus we have, same as \eqref{D11}, that
\begin{equation} \label{E12}
    \begin{split}
       \int_{|z|\geq R}\left(|z|^{\frac{\sigma-l-1}{2}} e^{-\langle \rho, z\rangle} e^{|\rho|\frac{|z|}{2}} e^{-\frac{\xi}{2}|z|}\right)^{p'}|z|^{-\alpha p'}dz
       \leq C
       <\infty.
    \end{split}
\end{equation}
Finally, by combining \eqref{E11} and \eqref{E12} we get
\begin{equation*}
    \mathcal{D}_1^*=\sup_{R > 0}\left(\int_{|z|\leq R}|z|^{-\beta q}dz\right)^{\frac{1}{q}}\left(\int_{|z|\geq R}\left(|z|^{\frac{\sigma-l-1}{2}} e^{-\langle \rho, z\rangle} e^{|\rho|\frac{|z|}{2}} e^{-\frac{\xi}{2}|z|}\right)^{p'}|z|^{-\alpha p'}dz\right)^{\frac{1}{p'}} \leq C<\infty.
\end{equation*}
Therefore, \eqref{I3q} and \eqref{I3q+} implies that
\begin{equation}
    I^{q}_{3}\leq C\left(\int_{X}|x|^{\alpha p}|u(x)|^{p}dx\right)^{\frac{q}{p}}.
\end{equation}

{\bf Step 3.} Let us now focus on the remaining case of $I_2.$
We  need to show that 
\begin{equation} \label{step3est}
    I_{2}=\left(\int_{ X}\left(\int_{B\left(0,2|x|\right)\setminus B\left(0,\frac{|x|}{2}\right)} |G_{\xi, \sigma}(y^{-1}x)u(y)|\,dy\right)^{q}|x|^{-\beta q}dx \right)^{\frac{1}{q}} \leq  C\||x|^{\alpha}u\|_{L^{p}(X)}.
\end{equation}
  We rewrite $I_{2}$ in the following form:
\begin{equation}
    \begin{split}
        I_{2}^{q}
        =\sum_{k=-\infty}^{+\infty}\int_{2^{k}\leq|x|\leq 2^{k+1}}\left(\int_{B(0,2|x|)\setminus B\left(0,\frac{|x|}{2}\right)}|G_{\xi, \sigma}(y^{-1}x)u(y)| dy\right)^{q}\frac{dx}{|x|^{\beta q}}.
    \end{split}
\end{equation} 
Since $|x|^{\beta q}$ is non-decreasing with respect to $|x|$ near the origin, there exists $k_0 \in \mathbb{Z}$ with $k_0 \leq -3$ such that $x \mapsto |x|^{\beta q}$ is non-decreasing for all $x$ satisfying $0< |x|<2^{k_0+1}.$ Thus, it makes sense to decompose $I_2^q$ into two parts as follows:
\begin{equation}
    I_2^q
    =I^{q}_{2,1}+I^{q}_{2,2},
\end{equation}
where 
\begin{equation}\label{I21}
    I^{q}_{2,1}=\sum_{k=-\infty}^{k_0}\int_{2^{k}\leq|x|\leq 2^{k+1}}\left(\int_{B(0,2|x|)\setminus B\left(0,\frac{|x|}{2}\right)}|G_{\xi, \sigma}(y^{-1}x)u(y)|dy\right)^{q}\frac{dx}{|x|^{\beta q}},
\end{equation}
and 
\begin{equation}
    I^{q}_{2,2}=\sum_{k=k_0+1}^{+\infty}\int_{2^{k}\leq|x|\leq 2^{k+1}}\left(\int_{B(0,2|x|)\setminus B\left(0,\frac{|x|}{2}\right)}|G_{\xi, \sigma}(y^{-1}x)u(y)|dy\right)^{q}\frac{dx}{|x|^{\beta q}}.
\end{equation}
First, we   show that $G_{\xi, 
\sigma}\in L^{r}(X)$ for $r \in [1, \infty]$ such that $1-\frac{1}{r}=\frac{1}{p}-\frac{1}{q},$ which will play a significant role in our proof. Indeed,
\begin{equation}\label{esti1}
    \begin{split}
       & \int_{X}|G_{\xi, \sigma}(x)|^r dx\\&=\int_{|x|\leq1} |G_{\xi, \sigma}(x)|^r dx+\int_{|x|>1}|G_{\xi, \sigma}(x)|^rdx\\&
        =\int_{|x|\leq 1}|x|^{r(\sigma-n)}dx+\int_{|x|>1}|x|^{\frac{r(\sigma-l-1)}{2}-r|\Sigma^{+}_0|}\varphi_{0}(x)^r e^{-\xi|x|r}dx\\&= \int_{\{H \in \overline{\mathfrak{a}^+}: |H|  \leq 1\}} |H|^{r(\sigma-n)} H^{n-l}dH+\int_{\{H \in \overline{\mathfrak{a}^+}: |H|  \geq 1\}}|H|^{\frac{r(\sigma-l-1)}{2}}e^{-r\langle\rho, H \rangle} e^{-\xi|H|r} e^{\langle\rho, H \rangle}dH\\&
        \lesssim \int_{0}^{1}s^{r(\sigma-n)+(n-1)}ds+\int_{\{H \in \overline{\mathfrak{a}^+}: |H|  \geq 1\}}|H|^{\frac{r(\sigma-l-1)}{2}}e^{-(r-1)\langle\rho, H \rangle} e^{-\xi|H|r} dH<\infty,
    \end{split}
\end{equation}
since $r(\sigma-n)+(n-1)>0$ by $\frac{\sigma-n-\alpha-\beta}{n}-\frac{1}{p}+\frac{1}{q}+1=0$ with $\alpha+\beta \geq 0$ and $r \in [1, \infty]$. 

Let us first estimate $I_{2,2}^q$. By using estimate \eqref{esti1}, Young's inequality for $1+\frac{1}{q}=
\frac{1}{r}+\frac{1}{p}$ with $1 \leq r \leq \infty$ and by setting  $\tilde{u}(y)=|y|^{\alpha}u(y)$, we establish 

\begin{equation}\label{I22vi}
\begin{split}
     I_{2,2}^{q}&=\sum_{k=-k_0+1}^{+\infty}\int_{2^{k}\leq|x|\leq 2^{k+1}}\left(\int_{B(0,2|x|)\setminus B\left(0,\frac{|x|}{2}\right)}|G_{\xi, \sigma}(y^{-1}x)u(y)|dy\right)^{q}\frac{dx}{|x|^{\beta q}}\\&
     =\sum_{k=-k_0+1}^{+\infty}\int_{2^{k}\leq|x|\leq 2^{k+1}}\left(\int_{B(0,2|x|)\setminus B\left(0,\frac{|x|}{2}\right)}|x|^{-\alpha}|x|^{\alpha}|G_{\xi, \sigma}(y^{-1}x)u(y)|dy\right)^{q}\frac{dx}{|x|^{\beta q}}\\&
     \leq C\sum_{k=-k_0+1}^{+\infty}\int_{2^{k}\leq|x|\leq 2^{k+1}}\left(\int_{B(0,2|x|)\setminus B\left(0,\frac{|x|}{2}\right)}|y|^{\alpha}|G_{\xi, \sigma}(y^{-1}x)u(y)|dy\right)^{q}\frac{dx}{|x|^{\alpha q+\beta q}}\\&
     \stackrel{\alpha +\beta\geq0}\leq C \sum_{k=-k_0+1}^{+\infty}\int_{2^{k}\leq|x|\leq 2^{k+1}}\left(\int_{B(0,2|x|)\setminus B\left(0,\frac{|x|}{2}\right)}|G_{\xi, \sigma}(y^{-1}x)|y|^{\alpha}u(y)|dy\right)^{q}dx\\&
     =C \sum_{k=-k_0+1}^{+\infty}\int_{2^{k}\leq|x|\leq 2^{k+1}}\left(\int_{B(0,2|x|)\setminus B\left(0,\frac{|x|}{2}\right)}|G_{\xi, \sigma}(y^{-1}x)\tilde{u}(y)|dy\right)^{q}dx.
                         \end{split}
\end{equation}
Note that from the conditions $2^{k} \leq |x| \leq 2^{k+1}$ and $|x| \leq 2|y| \leq 4|x|$ we deduce that $2^{k-1} \leq |y| \leq 2^{k+3}.$ Therefore, from \eqref{I22vi} we further have 
\begin{equation} \label{I22fin}
    \begin{split}
        I_{2,2}^q &
     \leq C\sum_{k=-k_0+1}^{+\infty}\int_{2^{k}\leq|x|\leq 2^{k+1}}\left(\int_{2^{k-1}\leq|y|\leq 2^{k+3}}|G_{\xi, \sigma}(y^{-1}x)\tilde{u}(y)|dy\right)^{q}dx\\&
     \leq C\sum_{k=-k_0+1}^{+\infty}\int_{X}\left(\int_{2^{k-1}\leq|y|\leq 2^{k+3}}|G_{\xi, \sigma}(y^{-1}x)|\,|\tilde{u}(y)|dy\right)^{q}dx\\&
     = C\sum_{k=-k_0+1}^{+\infty}\||G_{\xi, \sigma}|*(|\tilde{u}||\chi_{\{2^{k-1}\leq|\cdot|\leq 2^{k+3}\}})\|^{q}_{L^{q}(X)}\\&
     \leq C\sum_{k=-k_0+1}^{+\infty}\|G_{\xi, \sigma}\|^{q}_{L^{r}(X)}\|\tilde{u}\chi_{\{2^{k-1}\leq|\cdot|\leq 2^{k+3}\}}\|^{q}_{L^{p}(X)}\\&
     \leq C\|G_{\xi, \sigma}\|^{q}_{L^{r}(X)} \sum_{k \in \mathbb{Z}}\|\tilde{u}\chi_{\{2^{k-1}\leq|\cdot| \leq 2^{k+3}\}}\|^{q}_{L^{p}(X)} \\&\stackrel{\eqref{esti1}}\leq C \||x|^{\alpha}u\|_{L^p(X)}^q. 
    \end{split}
\end{equation}
Finally, we estimate $I_{2,1}^q.$ By combining the triangle inequality for the Riemannian distance with $|x|\leq2|y|\leq 4|x|,$ one can deduce that $|y^{-1}x|\leq |y|+|x|\leq 3|x|.$ Since $|x|\leq2|y|\leq 4|x|$ and $2^{k} \leq |x| \leq 2^{k+1}$   with $k \leq k_0$, it yields that
\begin{equation}\label{ocendisk3}
    |y^{-1}x|\leq 3 |x|\leq 3\cdot2^{k+1} \leq 3\cdot2^{k_0+1}<1,\,\,\text{for}\,k_0\in[-\infty,-3].
\end{equation}
Thus, by using \eqref{ocendisk3} with Young's inequality for $1+\frac{1}{q}= \frac{1}{r}+\frac{1}{p}$, we establish
\begin{equation}\label{I211}
\begin{split}
     I^{q}_{2,1}&=\sum_{k=-\infty}^{k_0}\int_{2^{k}\leq|x|\leq 2^{k+1}}\left(\int_{B(0,2|x|)\setminus B\left(0,\frac{|x|}{2}\right)}|G_{\xi, \sigma}(y^{-1}x)u(y)|dy\right)^{q}\frac{dx}{|x|^{\beta q}}\\&
     =\sum_{k=-\infty}^{k_0}\int_{2^{k}\leq|x|\leq 2^{k+1}}\left(\int_{B(0,2|x|)\setminus B\left(0,\frac{|x|}{2}\right)}|y|^{-\alpha}|y|^{\alpha}|G_{\xi, \sigma}(y^{-1}x)u(y)|dy\right)^{q}\frac{dx}{|x|^{\beta q}}dx\\&
    \leq C\sum_{k=-\infty}^{k_0}2^{(-\alpha-\beta)qk}\int_{2^{k}\leq|x|\leq 2^{k+1}}\left(\int_{B(0,2|x|)\setminus B\left(0,\frac{|x|}{2}\right)}|G_{\xi, \sigma}(y^{-1}x)\tilde{u}(y)|dy\right)^{q}dx\\&
    \leq C\sum_{k=-\infty}^{k_0}2^{(-\alpha-\beta)qk}\int_{2^{k}\leq|x|\leq 2^{k+1}}\left(\int_{2^{k-1}\leq 2|y|\leq 2^{k+2}}|G_{\xi, \sigma}(y^{-1}x)|\,|\tilde{u}(y)|dy\right)^{q}dx\\&
    \stackrel{\eqref{ocendisk3}}\leq C \sum_{k=-\infty}^{k_0}2^{(-\alpha-\beta)qk}\int_{X}\left(\int_{X}[|G_{\xi, \sigma}(y^{-1}x)|\chi_{\{0<|y^{-1}x|\leq 3\cdot 2^{k+1}\}}][|\tilde{u}(y)|\chi_{\{2^{k-1}\leq 2|y|\leq 2^{k+2}\}}]dy\right)^{q}dx \\& = C \sum_{k=-\infty}^{k_0}2^{(-\alpha-\beta)qk}\int_{X} |(|G_{\xi, \sigma}|\chi_{\{0<|\cdot|\leq 3\cdot 2^{k+1}\}}*|\tilde{u}|\chi_{\{2^{k-1}\leq 2|\cdot|\leq 2^{k+2}\}})(x)|^q dx \\& = C \sum_{k=-\infty}^{k_0}2^{(-\alpha-\beta)qk} \|(|G_{\xi, \sigma}|\chi_{\{0<|\cdot|\leq 3\cdot 2^{k+1}\}}*|\tilde{u}|\chi_{\{2^{k-1}\leq 2|\cdot|\leq 2^{k+2}\}})\|^q_{L^q(X)}\\&
    \leq C \sum_{k=-\infty}^{k_0}2^{(-\alpha-\beta)qk}\|G_{\xi, \sigma}\chi_{\{0<|\cdot|\leq 3\cdot 2^{k+1}\}}\|^{q}_{L^{r}(X)}\|\tilde{u}\chi_{\{2^{k-1}\leq 2|\cdot|\leq 2^{k+2}\}}\|^{q}_{L^{p}(X)}.
\end{split}
\end{equation}
Computing $\|G_{\xi, \sigma}\chi_{\{0<|\cdot|\leq 3\cdot 2^{k+1}\}}\|_{L^{r}(X)}^q$ with \eqref{ocendisk3}, we obtain
\begin{equation}\label{ocenGI2}
\begin{split}
    \|G_{\xi, \sigma}\chi_{\{0<|\cdot|\leq 3\cdot 2^{k+1}\}}\|_{L^{r}(X)}^q&= \left(\int_{\{x \in X: 0<|x|\leq 3\cdot 2^{k+1}\}}|G_{\xi, \sigma}(x)|^rdx\right)^{\frac{q}{r}}\\& \stackrel{3.2^{k+1}<1}= \left(\int_{\{H \in \overline{\mathfrak{a}^+}: 0<|H|\leq 3\cdot 2^{k+1}\}}|H|^{r(\sigma-n)} H^{n-l}dH\right)^{\frac{q}{r}} \\&
    \leq C \left( \int_{0}^{3\cdot 2^{k+1}}s^{r(\sigma-n)+n-1}ds\right)^{\frac{q}{r}}\\&
    \stackrel{r(\sigma-n)+n>0}\leq C 2^{\frac{(r(\sigma-n)+n) kq}{r}}.
\end{split}
\end{equation}
From assumptions $\frac{\sigma-n-\alpha -\beta}{n}-\frac{1}{p}+\frac{1}{q}+1=0$ and $1-\frac{1}{r}=\frac{1}{p}-\frac{1}{q}$, we have
\begin{equation}\label{sigmaalphabeta}
    (-\alpha-\beta)+(\sigma-n)+\frac{n}{r} = 0.
\end{equation}
Putting \eqref{ocenGI2} and \eqref{sigmaalphabeta} in \eqref{I211} and recalling that $k_0\leq -3$ we establish
\begin{equation}\label{I2fin2}
    \begin{split}
        I^{q}_{2,1}&\leq C \sum_{k=-\infty}^{k_0}2^{(-\alpha-\beta)qk}\|G_{\xi, \sigma}\chi_{\{0<|\cdot|\leq 3\cdot 2^{k+1}\}}\|^{q}_{L^{r}(X)}\|\tilde{u}\chi_{\{2^{k-1}\leq 2|\cdot|\leq 2^{k+2}\}}\|^{q}_{L^{p}(X)}\\&
        \stackrel{\eqref{ocenGI2}}\leq C\sum_{k=-\infty}^{k_0}2^{\left( (-\alpha-\beta)+(\sigma-n)+\frac{n}{r}\right)qk}\|\tilde{u}\chi_{\{2^{k-1}\leq 2|\cdot|\leq 2^{k+2}\}}\|^{q}_{L^{p}(X)}\\&
        \stackrel{\eqref{sigmaalphabeta}}\leq C \sum_{k=-\infty}^{k_0}\|\tilde{u}\chi_{\{2^{k-1}\leq 2|\cdot|\leq 2^{k+2}\}}\|^{q}_{L^{p}(X)}\\&
        \leq C\sum_{k=-\infty}^{+\infty}\|\tilde{u}\chi_{\{2^{k-1}\leq 2|\cdot|\leq 2^{k+2}\}}\|^{q}_{L^{p}(X)}\\&
        =C\|\tilde{u}\|^{q}_{L^{p}(X)}.
    \end{split}
\end{equation}
Then by putting together \eqref{I22fin} and \eqref{I2fin2}, we have
\begin{equation}
    I_{2}^{q}\leq I^{q}_{2,1}+I^{q}_{2,2}\leq C\|\tilde{u}\|^{q}_{L^{p}(X)}=C \||x|^{\alpha} u\|_{L^p(X)},
\end{equation}
completing the proof of \eqref{step3est}.
\end{proof}

\section{Hardy, Sobolev, Hardy-Sobolev, Gagliardo-Nirenberg and Caffarelli-Kohn-Nirenberg inequalities on symmetric spaces} \label{sec4}
In this section, we show the Hardy, Sobolev, Hardy-Sobolev, Gagliardo-Nirenberg and Caffarelli-Kohn-Nirenberg inequalities on symmetric spaces. First, we show the  Hardy-Sobolev inequality on symmetric spaces.
\begin{theorem}[Hardy-Sobolev inequality]\label{thmHardySobolev}
Let $X$ be a symmetric space of noncompact type  of dimension $n\geq3$ and rank $l\geq 1.$ Suppose that $0<\sigma<n$, $1<p\leq q<\infty$ and $0\leq \beta< \frac{n}{q}$ are such that $\frac{\sigma -\beta}{n}=\frac{1}{p}-\frac{1}{q}$.  Then for all  $u\in H^{\sigma, p}(X)$, we have
\begin{equation}\label{HardySobolev}
    \left\|\frac{u}{|x|^{\beta}}\right\|_{L^{q}(X)}\leq C\|u\|_{H^{\sigma, p}(X)},
\end{equation}
where  $C$ is a positive constant  independent of $u$.
\end{theorem}
\begin{cor}[Hardy inequality]\label{Hardyin}
For $q=p$  and $0< \sigma <\frac{n}{p} $   in Theorem \ref{thmHardySobolev}, the inequality \eqref{HardySobolev} gives the Hardy inequality on symmetric space, that is,
    \begin{equation}
    \left\|\frac{u}{|x|^{\sigma}}\right\|_{L^{p}(X)}\leq C\|u\|_{H^{\sigma, p}(X)}.
\end{equation}

\end{cor}\begin{cor}[Uncertainty principle]
 By  Theorem \ref{Hardyin}, we have the uncertainly principle on symmetric spaces of noncompact type, that is,
    \begin{equation}
    \left\|u\right\|_{L^{2}(X)}\leq C\|u\|_{H^{\sigma, p}(X)}\||x|^\sigma u\|_{L^{\frac{p}{p-1}}(X)}.
\end{equation}
\end{cor}
\begin{cor}[Sobolev inequality]
By taking $\beta=0$ in Theorem \ref{thmHardySobolev}, we obtain the Sobolev inequality on symmetric space of noncompact type, that is, for $0<\sigma<n$ and $1<p\leq q<\infty$ such that $\frac{\sigma }{n}=\frac{1}{p}-\frac{1}{q}$ we obtain
    \begin{equation}\label{Sobolevin}
    \left\|u\right\|_{L^{q}(X)}\leq C\|u\|_{H^{\sigma, p}(X)}.
\end{equation}

\end{cor}

\begin{proof}[Proof of Theorem \ref{thmHardySobolev}]
By \eqref{fund} and \eqref{equiso}, we have that, for $u \in H^{\sigma, p}(X)$ there is $f \in L^p(X)$ such that
\begin{equation}
    \begin{split}
       u= (-\Delta -|\rho|^{2}+\xi^{2})^{-\frac{\sigma}{2}}f=  G_{\xi, \sigma}*f.
    \end{split}
\end{equation}
Then by using this  with Theorem \ref{thmSteinWeiss},  we obtain
\begin{equation}
\begin{split}
     \left\|\frac{u}{|x|^{\beta}}\right\|^{q}_{L^{q}(X)}&=\int_{X}\left| (-\Delta -|\rho|^{2}+\xi^{2})^{-\frac{\sigma}{2}}f(x)\right|^{q}\frac{dx}{|x|^{\beta q}}\\&
     = \int_{X}|(G_{\xi, \sigma}*f)(x)|^{q}\frac{dx}{|x|^{\beta q}}\\&
     \leq C\|f\|^{q}_{L^{p}(X)}=C\|(-\Delta -|\rho|^{2}+\xi^{2})^{\frac{\sigma}{2}}u\|^{q}_{L^{p}(X)}
    \stackrel{\eqref{equiso}} \leq C\|u\|^{q}_{H^{\sigma,p}(X)},
\end{split}
\end{equation}
completing the proof.
\end{proof}
Let us now present the Gagliardo-Nirenberg inequality on symmetric space of noncompact type.
\begin{theorem}[Gagliardo-Nirenberg inequality]\label{thmGagliardoNirenberg}
Let $X$ be a symmetric space of noncompact type of dimension $n\geq3$ and rank $l\geq1$. Suppose  $0<\sigma<n$,  $\tau>0$, $p>1$, $\sigma p<n$, $\mu\geq1$, $a\in(0,1]$ and 
\begin{equation}
    \frac{1}{\tau}=a\left(\frac{1}{p}-\frac{\sigma}{n}\right)+\frac{(1-a)}{\mu}.
\end{equation}
Then there exists a positive constant $C$ such that we have 
\begin{equation}
        \|u\|_{L^{\tau}(X)}\leq C \|u\|^{a}_{H^{\sigma,p}}\|u\|^{1-a}_{L^{\mu}(X)},
\end{equation}
where $C$ is independent of $u.$
\end{theorem}
\begin{proof}
By using H\"{o}lder's inequality with $1=\frac{a\tau}{q}+\frac{(1-a)\tau}{\mu}$ where $\frac{1}{q}=\frac{1}{p}-\frac{\sigma}{n}$ and Sobolev inequality, we have
\begin{equation}
\begin{split}
        \|u\|^{\tau}_{L^{\tau}(X)}&= C \int_{X}|u(x)|^{a\tau}|u(x)|^{(1-a)\tau}dx\\&
        \leq \left(\int_{X}|u(x)|^{q}dx\right)^{\frac{a\tau}{q}}\left(\int_{X}|u(x)|^{\mu}dx\right)^{\frac{(1-a)\tau}{\mu}}\\&
        \stackrel{\eqref{Sobolevin}}\leq C\|u\|^{a\tau}_{H^{\sigma,p}(X)}\|u\|^{(1-a)\tau}_{L^{\mu}(X)},
\end{split}
\end{equation} completing the proof.
\end{proof}
In the next, we give a particular case of the Gagliardo-Nirenberg inequality, for that, we need the case $\mu=p=2$ and $\sigma=1$, which serves as a useful tool in our proof of the global existence of the wave equation in the next section.
\begin{cor}
Let $X$ be a symmetric space of noncompact type of dimension $n\geq3$ and rank $l\geq1$. Let $\tau\in\left[2,\frac{2n}{n-2}\right]$ and $a=\frac{n(\tau-2)}{\tau},$
then we have 
\begin{equation}\label{GNPDE}
    \|u\|_{L^{\tau}(X)}\leq C\|u\|^{a}_{H^{1,2}(X)}\|u\|^{1-a}_{L^{\tau}(X)}.
\end{equation}
\end{cor}
\begin{proof}
By taking $\sigma=1$, $p=2$ and $\nu=2$ in Theorem \ref{thmGagliardoNirenberg}, we get \eqref{GNPDE}.
\end{proof}
%\begin{rem}
%In Theorem \ref{thmGagliardoNirenberg}, by taking $a=1$, we get Sobolev inequality \eqref{Sobolevin}.
%\end{rem}
Then let us show the Caffarelli-Kohn-Nirenberg inequality in the following theorem:
\begin{theorem}
Let $X$ be a symmetric space of noncompact type of dimension $n\geq3$ and rank $l\geq1$ such that $0<\sigma<n$. Suppose    $p>1$, $0<q<\tau<\infty$ such that $a\in\left(\frac{\tau-q}{\tau},1\right]$ and $p\leq \frac{a\tau q}{q-(1-a)\tau}$. Let $b,c$ be real numbers such that $0\leq(c(1-a)-b)\leq \frac{n(q-(1-a)\tau)}{q\tau}$ and $\frac{\sigma-n}{n}-\frac{(c(1-a)-b)q}{an}+\frac{q-(1-a)\tau}{a\tau q}-\frac{1}{p}+1=0$. Then there exists a positive constant independent of $u$ such that
\begin{equation}
    \||x|^{b}u\|_{L^{\tau}(X)}\leq C\|u\|^{a}_{H^{\sigma, p}(X)}\||x|^{c}u\|^{1-a}_{L^{q}(X)}.
\end{equation}
\end{theorem}
\begin{proof}
By  H\"{o}lder's inequality with $\frac{q-(1-a)\tau}{q}+\frac{(1-a)\tau}{q}=1$ and Hardy-Sobolev inequality \eqref{HardySobolev} with assumptions of this theorem, we obtain
\begin{equation}
    \begin{split}
        \||x|^{b}u\|^{\tau}_{L^{\tau}(X)}&=\int_{X}|x|^{b\tau}|u(x)|^{\tau}dx\\&
        =\int_{X}\frac{|u(x)|^{a\tau}}{|x|^{(c(1-a)-b)\tau}}\frac{|u(x)|^{(1-a)\tau}}{|x|^{-c(1-a)\tau}}dx\\&
        \leq \left(\int_{X}\frac{|u(x)|^{a\tau\frac{q}{q-(1-a)\tau}}}{|x|^{(c(1-a)-b)\tau\frac{q}{q-(1-a)\tau}}}dx\right)^{\frac{q-(1-a)\tau}{q}}\left(\int_{X}\frac{|u(x)|^{q}}{|x|^{-cq}}dx\right)^{\frac{(1-a)\tau}{q}}\\&
        = \left(\int_{X}\frac{|u(x)|^{\frac{a\tau q}{q-(1-a)\tau}}}{|x|^{\frac{c(1-a)-b}{a}\frac{a\tau q}{q-(1-a)\tau}}}dx\right)^{a\tau \frac{q-(1-a)\tau}{a\tau q}}\left(\int_{X}\frac{|u(x)|^{q}}{|x|^{-cq}}dx\right)^{\frac{(1-a)\tau}{q}}\\&
        \stackrel{\eqref{HardySobolev}}\leq C\|u\|_{H^{\sigma,p}(X)}^{a\tau}\||x|^{c}u\|^{(1-a)\tau},
    \end{split}
\end{equation}
completing the proof.
\end{proof}
\section{Nonlinear wave equation with a damped term associated with the Laplace-Beltrami operator on symmetric spaces} \label{sec5}
In this section, we show small data global existence for the semilinear wave equation with damped term involving the Laplace-Beltrami operator on symmetric spaces of noncompact type. In \cite{RT} and \cite{VY19}, similar questions have been treated  for Rockland operators on graded Lie groups and for Dunkl Laplacian on Euclidean spaces. The strategy here follows that of \cite{RT}. The main aim of this section is to obtain the global existence result for the Cauchy problem involving the shifted Laplace-Beltrami operator $\Delta_x=\Delta+|\rho|^2.$ We consider the shifted Laplace-Beltrami operator instead of usual Laplace-Beltrami operator just to make computations clear and simple, otherwise the result still holds true for the usual Laplace-Beltrami operator.   
\begin{equation}\label{problem1}
\begin{split}
\begin{cases}
 u_{tt}(x,t)-\Delta_x u(x,t)+bu_{t}(x,t)+mu(x,t)=f(u)(x,t),\\
 u(x,0)=u_{0}(x),\,\,\,\,x\in X,\\
 u_{t}(x,0)=u_{1}(x),\,\,\,\,x\in X,
 \end{cases}
 \end{split}
\end{equation}
where $b,m>0$ and $f:\mathbb{R}\rightarrow \mathbb{R}$ satisfies the following conditions:
\begin{equation}\label{f1}
    f(0)=0,
\end{equation}
and
\begin{equation}\label{f2}
    |f(u)-f(v)|\leq C(|u|^{p-1}+|v|^{p-1})|u-v|.
\end{equation}
To show the global existence of a solution of Cauchy problem \eqref{problem1} we need the following lemma. 
\begin{lemma}\label{prepPDE}
Assume that $s\in\mathbb{R}$. Let $u$ be a solution of \eqref{problem1} with $f=0$, $u_{0}\in H^{s,2}$ and $u_{1}\in H^{s-1,2}$. Then there exists $\delta>0$ such that
\begin{equation}\label{ocensol}
    \|u\|^{2}_{H^{s,2}}\leq Ce^{-\delta t}\left(\|u_{0}\|^{2}_{H^{s,2}}+\|u_{1}\|^{2}_{H^{s-1,2}}\right),
\end{equation}
and 
\begin{equation}\label{ocensolt}
    \|u_{t}\|^{2}_{H^{s,2}}\leq Ce^{-\delta t}\left(\|u_{0}\|^{2}_{H^{s+1, 2}}+\|u_{1}\|^{2}_{H^{s,2}}\right).
\end{equation}
\end{lemma}
\begin{proof}
By applying the Helgason Fourier transform  in \eqref{problem1} with respect to the variable $x$ and using the fact that $\widehat{(\Delta_x u)}=-|\lambda|^2 \widehat{u}$, we obtain
\begin{equation}\label{problem2}
\begin{split}
\begin{cases}
 \hat{u}_{tt}+|\lambda|^{2}\hat{u}+b\hat{u}_{t}+m\hat{u}=0,\\
 \hat{u}(0)=\hat{u}_{0},\\
 \hat{u}_{t}(0)=\hat{u}_{1}.
 \end{cases}
 \end{split}
\end{equation}
Therefore, the  solution of \eqref{problem2} is given by
\begin{equation}\label{solution1}
\begin{split}
\hat{u}=
\begin{cases}
 e^{-\frac{b}{2}t}\left[\hat{u}_{0}\left(\cosh\frac{\sqrt{b^{2}-4\gamma^2}}{2}t+\frac{b}{\sqrt{b^{2}-4\gamma^2}}\sinh\frac{\sqrt{b^{2}-4\gamma^2 }}{2}t \right)+\frac{2}{\sqrt{b^{2}-4\gamma^2}}\hat{u}_{1}\sinh\frac{\sqrt{b^{2}-4\gamma^2}}{2}t\right],\,\,b>2\gamma,\\
 \ e^{-\frac{b}{2}t}\left[\hat{u}_{0}\left(1+\frac{b}{2}t\right)+t\hat{u}_{1}\right],\,\,b=2\gamma,\\
 e^{-\frac{b}{2}t}\left[\hat{u}_{0}\left(\cos\frac{\sqrt{4\gamma^2-b^{2}}}{2}t+\frac{b}{\sqrt{4\gamma^2-b^{2}}}\sin\frac{\sqrt{4\gamma^2-b^{2}}}{2}t \right)+\frac{2}{\sqrt{4\gamma^2-b^{2}}}\hat{u}_{1}\sin\frac{\sqrt{4\gamma^2-b^{2}}}{2}t\right],\,\,b<2\gamma,
 \end{cases}
 \end{split}
\end{equation}
where $\gamma=\sqrt{m+|\lambda|^{2}}$.

We set $B:=b^2- 4 \gamma^2.$ We will now deduce a pointwise estimate for each case. 
First, let us consider the case when $B>0$.
We know that
\begin{equation}
    \frac{1-e^{-t}}{t}\leq 1,\,\,\,\,t\geq0,
\end{equation}
and for any $c>0$, we have
\begin{equation}
    bct\leq e^{bct},\,\,\,\,t\geq0,
\end{equation}
so that 
\begin{equation}
   (bct) e^{-\frac{bt}{2}+\frac{\sqrt{b^{2}-4\gamma^2}}{2}}\leq e^{-\frac{b(1-c)t}{2}+\frac{\sqrt{b^{2}-4\gamma^2}}{2}t}\leq e^{-\frac{b(1-c)t}{2}+\frac{\sqrt{b^{2}-4m}}{2}t},
\end{equation}
thanks to $b>2\gamma=2\sqrt{m+|\lambda|^{2}}>2\sqrt{m}.$
Thus, if we choose $0<c<1-\frac{\sqrt{b^{2}-4m}}{b}$, we get 
\begin{equation}\label{estexp}
    (bct) e^{-\frac{bt}{2}+\frac{\sqrt{b^{2}-4\gamma^2}}{2}}\leq e^{-\frac{b(1-c)t}{2}+\frac{\sqrt{b^{2}-4m}}{2}t}\leq e^{-\delta t}.
\end{equation}
Therefore, by using these facts we establish
\begin{align} \label{fest}
 |\hat{u}|&=\big{|}e^{-\frac{b}{2}t}\big{[}\hat{u}_{0}\left(\cosh\frac{\sqrt{b^{2}-4\gamma^2}}{2}t+\frac{b}{\sqrt{b^{2}-4\gamma^2}}\sinh\frac{\sqrt{b^{2}-4\gamma^2}}{2}t \right) \nonumber \\&
        \quad\quad+\frac{2}{\sqrt{b^{2}-4\gamma^2}}\hat{u}_{1}\sinh\frac{\sqrt{b^{2}-4\gamma^2}}{2}t\big{]}\big{|}\nonumber\\&
        \leq  Ce^{-\frac{b}{2}t-\frac{\sqrt{b^{2}-4\gamma^2}}{2}t}\left(|\hat{u}_{0}|\left(1+\frac{bt}{2}\right)+\frac{1}{\sqrt{b^{2}-4\gamma^2}}|\hat{u}_{1}|\right)\nonumber\\&
        \leq Ce^{-\delta t}\left(|\hat{u}_{0}|+\frac{1}{\sqrt{b^{2}-4\gamma^2}}|\hat{u}_{1}|\right).\end{align}

For the case $B=0$, that is, $b=2\gamma,$ by using previous computation in \eqref{estexp}, we get \begin{equation} \label{lab5:13}
    \begin{split}
        |\hat{u}|&=\left|e^{-\frac{b}{2}t}\left[\hat{u}_{0}\left(1+\frac{b}{2}t\right)+t\hat{u}_{1}\right]\right|
        \leq  Ce^{-\delta t}\left(|\hat{u}_{0}|+|\hat{u}_{1}|\right).
    \end{split}
\end{equation}

Finally, for case when $B<0,$ that is, $b<2\gamma,$ we perform a similar computation as in the case $b>2\gamma$ to obtain
\begin{equation}
    \begin{split}
        |\hat{u}|&=\big{|}e^{-\frac{b}{2}t}\big{[}\hat{u}_{0}\left(\cos\frac{\sqrt{4\gamma^2 -b^{2}}}{2}t+\frac{b}{\sqrt{4\gamma^2-b^{2}}}\sin\frac{\sqrt{4\gamma-b^{2}}}{2}t \right)\\&+\frac{2}{\sqrt{4\gamma^2-b^{2}}}\hat{u}_{1}\sin\frac{\sqrt{4\gamma^2-b^{2}}}{2}t\big{]}\big{|}\\&
        \leq Ce^{-\delta t}\left(|\hat{u}_{0}|+\frac{1}{\sqrt{4\gamma^2-b^{2}}}|\hat{u}_{1}|\right).
    \end{split}
\end{equation}
To get the Sobolev norm estimates of solutions, we consider again case by case using the previously established pointwise estimates. We begin with the case when $b^2<4m,$ that is, $B<0$ for all $\lambda \in \mathfrak{a}^*$.  In this case, the function $\frac{\xi^2+|\lambda|^{2}}{4m-b^{2}+|\lambda|^{2}}$ is bounded for all $\lambda \in \mathfrak{a}^*$ as $b^2<4m.$ Using the Plancherel formula for the Helgason-Fourier transform, we have 
\begin{equation}
    \begin{split}
        \|u\|^{2}_{H^{s,2}}&=\frac{1}{|W|}\int_{\mathfrak{a}^*\times K}(\xi^2+|\lambda|^{2})^{s}|\hat{u}|^{2}\frac{d\lambda dk}{|c(\lambda)|^{2}}\\&
        \leq C e^{-2\delta t}\int_{\mathfrak{a}^*\times K}(\xi^2+|\lambda|^{2})^{s}\left(|\hat{u}_{0}|^{2}+\frac{1}{4\gamma^2-b^{2}}|\hat{u}_{1}|^{2}\right)\frac{d\lambda dk}{|c(\lambda)|^{2}}\\&
       =  C e^{-2\delta t}\int_{\mathfrak{a}^*\times K}\left((\xi^2+|\lambda|^{2})^{s}|\hat{u}_{0}|^{2}+\frac{(\xi^2+|\lambda|^{2})^{s}}{4m-b^{2}+4|\lambda|^{2}}|\hat{u}_{1}|^{2}\right)\frac{d\lambda dk}{|c(\lambda)|^{2}}\\&
       \leq Ce^{-2\delta t}\int_{\mathfrak{a}^*\times K}\left((\xi^2+|\lambda|^{2})^{s}|\hat{u}_{0}|^{2}+(\xi^2+|\lambda|^{2})^{s-1}|\hat{u}_{1}|^{2}\right)\frac{d\lambda dk}{|c(\lambda)|^{2}}\\&
       =Ce^{-2\delta t}\|u_{0}\|^{2}_{H^{s,2}}+Ce^{-2\delta t}\|u_{1}\|^{2}_{H^{s-1,2}}.
        \end{split}
\end{equation}
The case $4m \leq b^2$ is a bit tricky due to the presence of singularity at $|\lambda|= \frac{\sqrt{b^2-2m}}{2}.$ To tackle this situation, we divide the integral over the set $D=\{\lambda \in \mathfrak{a}^*: |B|<1\}$ and $D^C:=\mathfrak{a}^* \backslash D.$  In other words, we will calculate
\begin{align} \label{f3}
    \|u\|^{2}_{H^{s,2}}&=\frac{1}{|W|}\int_{D\times K}(\xi^2+|\lambda|^{2})^{s}|\hat{u}|^{2}\frac{d\lambda dk}{|c(\lambda)|^{2}}+\frac{1}{|W|}\int_{D^C\times K}(\xi^2+|\lambda|^{2})^{s}|\hat{u}|^{2}\frac{d\lambda dk}{|c(\lambda)|^{2}}.
\end{align}
On $\mathfrak{a}^* \times K,$ we are always in the case when $B<0$ or $B>0.$ Therefore, one can deduce the same estimate \eqref{fest}. Indeed, by the similar approach, since the function 
$\lambda \mapsto \frac{\xi^2+|\lambda|^{2}}{|4m-b^{2}+|\lambda|^{2}|}$ is bounded on $D^C,$ we obtain that
\begin{align} \label{ff2}
   \frac{1}{|W|} \int_{D^C\times K}(\xi^2+|\lambda|^{2})^{s}|\hat{u}|^{2}\frac{d\lambda dk}{|c(\lambda)|^{2}}   \leq Ce^{-2\delta t}\int_{D^C \times K}\left((\xi^2+|\lambda|^{2})^{s}|\hat{u}_{0}|^{2}+(\xi^2+|\lambda|^{2})^{s-1}|\hat{u}_{1}|^{2}\right)\frac{d\lambda dk}{|c(\lambda)|^{2}}.
\end{align}
Next, for the integral over the set $D \times K.$ We can see that the estimate \eqref{lab5:13} holds in general situation for some constants $C, \delta>0$ by analysing it case by case. For the case, when $B=0$ has been proved in \eqref{lab5:13}. For the case $B>0$ and $B<0,$ we showed that the estimate \eqref{fest} holds, which by adjusting the constant will imply \eqref{lab5:13} for these two cases as well. In conclusion, we have the following estimate 
\begin{align}
    \frac{1}{|W|}\int_{D\times K}(\xi^2+|\lambda|^{2})^{s}|\hat{u}|^{2}\frac{d\lambda dk}{|c(\lambda)|^{2}}  \leq Ce^{-2\delta t}\int_{D \times K}\left((\xi^2+|\lambda|^{2})^{s}|\hat{u}_{0}|^{2}+(\xi^2+|\lambda|^{2})^{s}|\hat{u}_{1}|^{2}\right)\frac{d\lambda dk}{|c(\lambda)|^{2}}.
\end{align}
Now, on the set $D,$ we have $|B|<1,$ which gives that 
$$4 |\lambda|^2\leq |4|\lambda|^2+4m-b^2|+b^2-4m<b^2-4m+1,$$
which says that on the set $D \times K,$ the quantity $\xi^2+|\lambda|^2$ is bounded above by a constant independent of $t.$ Thus we get the estimate
\begin{align} \label{ff1}
    \frac{1}{|W|}\int_{D\times K}(\xi^2+|\lambda|^{2})^{s}|\hat{u}|^{2}\frac{d\lambda dk}{|c(\lambda)|^{2}}  \leq Ce^{-2\delta t}\int_{D \times K}\left((\xi^2+|\lambda|^{2})^{s}|\hat{u}_{0}|^{2}+(\xi^2+|\lambda|^{2})^{s-1}|\hat{u}_{1}|^{2}\right)\frac{d\lambda dk}{|c(\lambda)|^{2}}.
\end{align}
Now, using \eqref{f2} and \eqref{f1} in \eqref{f3}, we obtained the desired estimate \eqref{ocensol}.

To obtain \eqref{ocensolt}, we need take the  derivative with respect to the variable $t$ of \eqref{solution1} and repeat similar computations and estimates as above. Hence, the proof is completed. 
\end{proof}
Here we show the global existence and uniqueness result for the solution of problem \eqref{problem1}.
\begin{theorem} Let $X$ be a symmetric space of noncompact type  of dimension $n\geq3$. Let $1\leq p\leq \frac{n}{n-2}$. Suppose that $f$ satisfies the conditions \eqref{f1}-\eqref{f2}. Assume that  $u_{0}\in H^{1,2}(X)$ and $u_{1}\in L^{2}(X)$ are such that $\|u_{0}\|_{H^{1,2}(X)}+\|u_{1}\|_{L^{2}(X)}<\varepsilon$. Then there exists $\varepsilon_{0}>0$ such that $0<\varepsilon\leq\varepsilon_{0}$ the Cauchy problem \eqref{problem1} has a unique global solution  $u\in C(\mathbb{R}_{+}, H^{1,2}(X))\cap C^{1}(\mathbb{R}_{+},L^{2}(X))$.
\end{theorem}
\begin{proof}
The proof of this theorem is based Banach fixed point theorem. First, let us consider the following Banach space
\begin{equation}
    Z:=\{u:u\in C(\mathbb{R}_{+}, H^{1,2}(X))\cap C^{1}(\mathbb{R}_{+},L^{2}(X)), \|u\|_{Z}\leq M\},
\end{equation}
where $M$ will be defined later and the  norm is defined by

\begin{equation}
    \|u\|_{Z}:=\sup_{t>0}(1+t)^{-\frac{1}{2}}e^{\delta t}\left(\|\Delta^{\frac{1}{2}}u(\cdot,t)\|_{L^{2}(X)}+\|u(\cdot,t)\|_{L^{2}(X)}+\|u_{t}(\cdot,t)\|_{L^{2}(X)}\right),
\end{equation}
where $\delta>0$ as in Lemma \ref{prepPDE}. Let us define the mapping $L:Z\rightarrow Z$, 
\begin{equation}
    (Lu)(x,t)=\psi(x,t)+\int_{0}^{t}[K(f(u))](x,t-s)ds,
\end{equation}
where $\psi$ is the solution of the linear problem $\eqref{problem1}$ and $Kv$ is the solution of the following Cauchy problem:
\begin{equation}\label{problem3}
\begin{split}
\begin{cases}
 u_{tt}(x,t)-\Delta_{x}u(x,t)+bu_{t}(x,t)+mu(x,t)=0,\\
 u(x,0)=0,\,\,\,\,x\in X,\\
 u_{t}(x,0)=v(x),\,\,\,\,x\in X.
 \end{cases}
 \end{split}
\end{equation}
For simplicity, let us denote 
$J(f(u))(x,t):=\int_{0}^{t}[Kf(u)](x,t-s)ds.$
Firstly, let us show that the 
\begin{equation}\label{LuLv}
    \|Lu-Lv\|_{Z}\leq c\|u-v\|_{Z}, \,\,\,\text{for\,\,\,all}\,\,\,u,v\in Z,
\end{equation}
where $c\in(0,1)$.

For all $t>0$, $u,v\in Z$ and by using  \eqref{f1}, H\"{o}lder's inequality with $\frac{2(p-1)}{2p}+\frac{2}{2p}=1$, Gagliardo-Nirenberg's inequality \eqref{GNPDE}, Young's inequality and equivalence of the norms, we compute 
\begin{equation}\label{focen1}
\begin{split}
    &\|f(u)(\cdot,t)-f(v)(\cdot,t)\|^{2}_{L^{2}(X)}\leq\int_{X}|f(u)(x,t)-f(v)(x,t)|^{2}dx\\&
    \stackrel{\eqref{f1}}\leq C\int_{X}\left[|u(x,s)|^{2(p-1)}+|v(x,s)|^{2(p-1)}\right]|u(x,s)-v(x,s)|dx\\&
    \stackrel{HI}\leq  C\left[\|u(\cdot,t)\|^{2(p-1)}_{L^{2p}(X)}+\|v(\cdot,t)\|^{2(p-1)}_{L^{2p}(X)}\right]\|u(\cdot,t)-v(\cdot,t)\|^{2}_{L^{2p}(X)}\\&
    \stackrel{\eqref{GNPDE}}\leq C\left[\|u\|^{2(p-1)a}_{H^{1,2}(X)}\|u\|_{L^{2}(X)}^{2(p-1)(1-a)}+\|v\|^{2(p-1)a}_{H^{1,2}(X)}\|v\|_{L^{2}(X)}^{2(p-1)(1-a)}\right]\|u-v\|^{2a}_{H^{1,2}(X)}\|u-v\|_{L^{2}(X)}^{2(1-a)}\\&
    \stackrel{YI}\leq C \left[\|u\|^{2(p-1)}_{H^{1,2}(X)}+\|u\|_{L^{2}(X)}^{2(p-1)}+\|v\|^{2(p-1)}_{H^{1,2}(X)}+\|v\|_{L^{2}(X)}^{2(p-1)}\right]\left[\|u-v\|^{2}_{H^{1,2}(X)}+\|u-v\|_{L^{2}(X)}^{2}\right]\\&
\leq C\left[\|\Delta^{\frac{1}{2}}u\|^{2(p-1)}_{L^{2}(X)}+\|u\|_{L^{2}(X)}^{2(p-1)}+\|\Delta^{\frac{1}{2}}v\|^{2(p-1)}_{L^{2}(X)}+\|v\|_{L^{2}(X)}^{2(p-1)}\right]\\&
\times\left[\|\Delta^{\frac{1}{2}}(u-v)\|^{2}_{L^{2}(X)}+\|u-v\|_{L^{2}(X)}^{2}\right]\\&
\leq C\left[\left(\|\Delta^{\frac{1}{2}}u\|_{L^{2}(X)}+\|u\|_{L^{2}(X)}\right)^{2(p-1)}+\left(\|\Delta^{\frac{1}{2}}v\|_{L^{2}(X)}+\|v\|_{L^{2}(X)}\right)^{2(p-1)}\right]\\&
\times\left(\|\Delta^{\frac{1}{2}}(u-v)\|_{L^{2}(X)}+\|u-v\|_{L^{2}(X)}\right)^{2}.
\end{split}
\end{equation}
By using definition of the norm on the space $Z$, for all $u\in Z$ we have
\begin{equation}
    (1+t)^{-\frac{1}{2}}e^{\delta t}(\|\Delta^{\frac{1}{2}}u\|_{L^{2}(X)}+\|u\|_{L^{2}(X)})\leq \|u\|_{Z}\leq M,
\end{equation}
then we get 
\begin{equation}
    \|\Delta^{\frac{1}{2}}u\|_{L^{2}(X)}+\|u\|_{L^{2}(X)}\leq (1+t)^{\frac{1}{2}}e^{-\delta t}M.
\end{equation}
Putting this estimate in \eqref{focen1}, we get
\begin{equation}\label{focen2}
\begin{split}
    \|f(u)(\cdot,t)-f(v)(\cdot,t)\|^{2}_{L^{2}(X)}
&\leq C\left[\left(\|\Delta^{\frac{1}{2}}u\|_{L^{2}(X)}+\|u\|_{L^{2}(X)}\right)^{2(p-1)}+\left(\|\Delta^{\frac{1}{2}}v\|_{L^{2}(X)}+\|v\|_{L^{2}(X)}\right)^{2(p-1)}\right]\\&
\times\left(\|\Delta^{\frac{1}{2}}(u-v)\|_{L^{2}(X)}+\|u-v\|_{L^{2}(X)}\right)^{2}\\&
\leq C(1+t)^{p-1}e^{-2\delta (p-1)t}M^{2p-2}\left(\|\Delta^{\frac{1}{2}}(u-v)\|_{L^{2}(X)}+\|u-v\|_{L^{2}(X)}\right)^{2}\\&
\leq C(1+t)^{p}e^{-2\delta pt}M^{2p-2}\|u-v\|^{2}_{Z}.
\end{split}
\end{equation}

for all $t>0$ and $u,v\in Z$. 

Then, via the equivalence of the norms on $H^{1,2}$-spaces,  we have
\begin{equation}\label{Lnorm}
\begin{split}
   &\|Lu-Lv\|_{L^{2}(X)}+\|\Delta^{\frac{1}{2}}(Lu-Lv)\|_{L^{2}(X)}+\left\|\frac{\partial}{\partial t}(Lu-Lv)\right\|_{L^{2}(X)}\\&=\|J(f(u))-J(f(v))\|_{L^{2}(X)}+\|\Delta^{\frac{1}{2}}(J(f(u))-J(f(v)))\|_{L^{2}(X)}\\&
    \quad\quad\quad+\left\|\frac{\partial}{\partial t}\left[J(f(u))-J(f(v))\right]\right\|_{L^{2}(X)}\\&
    \leq C\|J(f(u))-J(f(v))\|_{H^{1,2}(X)}+\left\|\frac{\partial}{\partial t}\left[J(f(u))-J(f(v))\right]\right\|_{L^{2}(X)}.
\end{split}
\end{equation}
Let us now compute the first norm that appears in the right hand side of the last estimate. First, we note that 
$$Kf(u)(x,t)-Kf(v)(x,t)=K(f(u)-f(v))(x,t).$$
By combining the above equality with the Cauchy-Schwarz inequality, Lemma \ref{prepPDE} and \eqref{focen2},  we obtain 
\begin{equation}\label{ocen55}
\begin{split}
     \|J(f(u))-J(f(v))\|^{2}_{H^{1,2}(X)}&= \int_{X}\left|\int_{0}^{t}(\xi^2-\Delta_x)^{\frac{1}{2}}[K(f(u))-K(f(v))]ds\right|^{2}dx\\&
     \stackrel{C-S\,I}\leq  Ct\int_{0}^{t}\int_{X}\left|(\xi^2-\Delta_x)^{\frac{1}{2}}[K(f(u))-K(f(v))]\right|^{2}dxds\\&
     =Ct\int_{0}^{t}\int_{X}\left|(\xi^2-\Delta_x)^{\frac{1}{2}}K[f(u)-f(v)]\right|^{2}dxds\\&
     \stackrel{\eqref{ocensol}}\leq Cte^{-2\delta t}\int_{0}^{t}e^{2\delta s}\|f(u)-f(v)\|^{2}_{L^{2}(X)}ds\\&
     \stackrel{\eqref{focen2}}\leq C M^{2p-2}te^{-2\delta t}\int_{0}^{t}(1+s)^{p}e^{2\delta s}e^{-2\delta p}\|u-v\|^{2}_{Z}ds\\&
     =C M^{2p-2}te^{-2\delta t}\|u-v\|^{2}_{Z}\int_{0}^{t}(1+s)^{p}e^{-2\delta(p-1)s}ds\\&
     \leq CM^{2p-2}te^{-2\delta t}\|u-v\|^{2}_{Z}.
\end{split}
\end{equation}
We know that $Kf(u)(x,t)$ is a solution of  \eqref{problem3}, then $K(f(u))(x,0)=0$. By using this fact, we have
\begin{equation}
    \frac{\partial}{\partial t}J(f(u))(x,t-s)= \frac{\partial}{\partial t}\int_{0}^{t}[K(f)](x,t-s)ds=\int_{0}^{t}\frac{\partial}{\partial t}[K(f)](x,t-s)ds,
\end{equation}
and therefore, similarly to \eqref{ocen55}, with the use of Lemma \ref{prepPDE} (with $s=0$) and \eqref{focen2}, we obtain
\begin{equation}\label{ocen555}
\begin{split}
     \left\|\frac{\partial}{\partial t}(J(f(u))-J(f(v)))\right\|^{2}_{L^{2}(X)}&= \int_{X}\left|\int_{0}^{t}\frac{\partial}{\partial t}[K(f(u))-K(f(v))]ds\right|^{2}dx\\&
     \leq \int_{X}t\int_{0}^{t}\left|\frac{\partial}{\partial t}[K(f(u))-K(f(v))]\right|^{2}dsdx\\&
     = Ct\int_{0}^{t}\int_{X}\left|\frac{\partial}{\partial t}[K(f(u))-K(f(v))]\right|^{2}dxds\\&
     = Ct\int_{0}^{t}\int_{X}\left|\frac{\partial}{\partial t}K[f(u)-f(v)]\right|^{2}dxds\\&
     \stackrel{\eqref{ocensolt}}\leq Cte^{-2\delta t}\int_{0}^{t}e^{2\delta s}\|f(u)-f(v)\|^{2}_{L^{2}(X)}ds\\&
     \stackrel{\eqref{focen2}}\leq C M^{2p-2}te^{-2\delta t}\int_{0}^{t}(1+s)^{p}e^{2\delta s}e^{-2\delta p}\|u-v\|^{2}_{Z}ds\\&
     =C M^{2p-2}e^{-2\delta t}t\|u-v\|^{2}_{Z}\int_{0}^{t}(1+s)^{p}e^{-2\delta(p-1)s}ds\\&
     \leq CM^{2p-2}te^{-2\delta t}\|u-v\|^{2}_{Z}.
\end{split}
\end{equation}
Putting \eqref{ocen55} and \eqref{ocen555} in \eqref{Lnorm}, we have
\begin{equation}
\begin{split}
    &\|Lu-Lv\|_{L^{2}(X)}+\|\Delta^{\frac{1}{2}}(Lu-Lv)\|_{L^{2}(X)}+\left\|\frac{\partial}{\partial t}(Lu-Lv)\right\|_{L^{2}(X)}\\&
    \leq C\|J(f(u))-J(f(v))\|_{H^{1,2}(X)}+\left\|\frac{\partial}{\partial t}\left[J(f(u))-J(f(v))\right]\right\|_{L^{2}(X)}\\&
    \stackrel{\eqref{ocen55},\eqref{ocen555}}\leq CM^{p-1}t^{\frac{1}{2}}e^{-\delta t}\|u-v\|_{Z}.
\end{split}
\end{equation}
From the last fact we establish,
\begin{equation}
\begin{split}
   &(1+t)^{-\frac{1}{2}}e^{\delta t}\left[\|Lu-Lv\|_{L^{2}(X)} +\|\Delta^{\frac{1}{2}}(Lu-Lv)\|_{L^{2}(X)}+\left\|\frac{\partial}{\partial t}(Lu-Lv)\right\|_{L^{2}(X)}\right]\\&
   \leq C(1+t)^{-\frac{1}{2}}e^{\delta t}M^{p-1}t^{\frac{1}{2}}e^{-\delta t}\|u-v\|_{Z}\\&
   =C\left(\frac{t}{1+t}\right)^{\frac{1}{2}}M^{p-1}\|u-v\|_{Z}\\&
   \leq CM^{p-1}\|u-v\|_{Z}\\&
   =C_{1}M^{p-1}\|u-v\|_{Z},
   \end{split}
\end{equation}
and by taking supremum, we get
\begin{equation}\label{lastLuLv}
  \|Lu-Lv\|_{Z}\leq C_{1}M^{p-1}\|u-v\|_{Z}.
\end{equation}
By choosing $M^{p-1}=\frac{c}{C_{1}}$, where $c\in(0,1)$, we show \eqref{LuLv}, it means
\begin{equation}\label{ocenc1}
    \|Lu-Lv\|_{Z}\leq c\|u-v\|_{Z}.
\end{equation}

After that, let us prove 
\begin{equation}
    \|Lu\|_{Z}\leq M,\,\,\,\,\,\forall u\in Z.
\end{equation}

By taking $v=0$ and $f(0)=0$ in \eqref{lastLuLv}, we get 
\begin{equation}
\|Lu-L0\|_{Z}=\|Lu-\psi\|_{Z}\leq C_{1}M^{p-1}\|u\|_{Z}\leq C_{1}M^{p}.
\end{equation}
By Lemma \ref{prepPDE} and assumption of this theorem , we have 
\begin{equation}
    \|\psi\|_{Z}\leq C_{2}(\|u_{0}\|_{H^{1,2}(X)}+\|u_{1}\|_{L^{2}(X)})\leq C_{2}\varepsilon.
\end{equation}
By combining these facts with choosing $\varepsilon\leq \frac{M(1-c)}{C_{2}}=\varepsilon_{0}$ where $c$ is defined in \eqref{ocenc1}, we get
\begin{equation}
\begin{split}
    \|Lu\|_{Z}=\|Lu-\psi+\psi\|_{Z}&\leq \|Lu-\psi\|_{Z}+\|\psi\|_{Z}\\&
    \leq C_{1}M^{p}+C_{2}\varepsilon\\&
    \leq cM+C_{2}\varepsilon_{0}\\&
    =cM+C_{2}\frac{M(1-c)}{C_{2}}\\&
    =M,
\end{split}
\end{equation}
completing the proof the $L$ is a contraction on $Z.$ Therefore, Banach fixed point theorem guarantee ensure the existence of a global solution.

The uniqueness of the solution can be easily deduced using the above estimates. Indeed, suppose that there are two solutions $u,v \in Z$ of the Cauchy problem \eqref{problem1}. Set $\eta=u-v$ and fix $t^*>0.$ Since $u$ and $v$ are fixed point of $L$, using the previous computations \eqref{focen2}, \eqref{ocen55} and \eqref{ocen555} we obtain that
\begin{align} \label{groa}
 \nonumber   \|\eta(t, \cdot)\|_{H^{1,2}(X)}^2 +\|\partial_t \eta(t, \cdot)\|_{L^2(X)}^2 &= \big\| J(f(u))(t, \cdot)-J(f(v))(t, \cdot) \big\|_{H^{1,2}(X)}^2\\& \nonumber\quad +\big\|\partial_t (J(f(u))-J(f(v)))(t, \cdot)\big\|_{L^2(X)}^2\\ \nonumber&\leq C t e^{-2 \delta t} \int_0^t e^{2 \delta s} \|f(u(s, \cdot))-f(v(s, \cdot))\|_2^2 ds \\& \leq C t e^{-2 \delta t} \int_0^t \left( \|u(s, \cdot)\|_{H^{1,2}(X)}^{2(p-1)}+\|v(s, \cdot)\|_{H^{1,2}(X)}^2  \right) \|\eta(s, \cdot)\|_{H^{1,2}(X)}^2 ds. 
\end{align}
So, from the definition of $Z$ and using the equivalence of the norms \eqref{equiso} we see that  the map $t \mapsto \|u(s, \cdot)\|_{H^{1,2}(X)}^{2(p-1)}+\|v(s, \cdot)\|_{H^{1,2}(X)}^2  $  is continuous and therefore, it is bounded on the compact set $[0, t^*].$ So inequality \eqref{groa} takes the form 
$$\|\eta(t, \cdot)\|_{H^{1,2}(X)}^2 +\|\partial_t \eta(t, \cdot)\|_{L^2(X)}^2\leq C t e^{-2 \delta t} \int_0^t \|\eta(s, \cdot)\|_{H^{1,2}(X)}^2+\|\partial_t \eta(t, \cdot)\|_{L^2(X)}^2 ds,$$

and, so by Gronwall's lemma, we obtain 
$$\|\eta(t, \cdot)\|_{H^{1,2}(X)}^2 +\|\partial_t \eta(t, \cdot)\|_{L^2(X)}^2=0,$$ for all $0<t<t^*,$ which eventually shows that $\eta=0$ in $[0, t^*] \times X.$ Finally, since $t^*$ is arbitrary it implies that $\eta \equiv 0$ and hence $u=v.$ This completes the proof of the theorem. 
 \end{proof}

\end{document}